\PassOptionsToPackage{sort&compress}{natbib}
\documentclass{elsarticle}
\usepackage[utf8]{inputenc}
\usepackage[T1]{fontenc}
\usepackage{hyperref}

\makeatletter
\def\ps@pprintTitle{%
 \let\@oddhead\@empty
 \let\@evenhead\@empty
 \def\@oddfoot{}%
 \let\@evenfoot\@oddfoot}
\makeatother

\usepackage{
    amsmath,
    amsthm,
    amssymb, 
    hyperref, 
    tikz, 
    xcolor, 
    mathabx,
    framed
}
\usetikzlibrary {positioning}
\usetikzlibrary{fadings}
\usetikzlibrary{decorations.pathreplacing}
\usetikzlibrary{decorations.pathmorphing}
\tikzset{snake it/.style={decorate, decoration=snake}}
\newcommand{\dir}[1]{\overrightarrow{#1}}
\definecolor{gold}{RGB}{255,215,0}
\definecolor{softBlack}{RGB}{45, 47, 49}
\definecolor{creamWhite}{RGB}{245,244,241}
\definecolor{softGray}{RGB}{220, 216, 214}
\definecolor{brick}{RGB}{232, 48, 48}

\newcommand{\typesetoperator}[1] { \mathbf{#1} }
\DeclareMathOperator{\tw}{\typesetoperator{tw}}
\DeclareMathOperator{\rk}{\typesetoperator{rk}}
\DeclareMathOperator{\rw}{\typesetoperator{rw}}
\DeclareMathOperator{\bw}{\typesetoperator{bw}}
\DeclareMathOperator{\bcrk}{\typesetoperator{bcrk}}
\DeclareMathOperator{\dbw}{\typesetoperator{dbw}}

\DeclareMathOperator{\reach}{\typesetoperator{Reach^2}}
\DeclareMathOperator{\dred}{\typesetoperator{DRED^2}}
\DeclareMathOperator{\MSO}{\typesetoperator{MSO}}
\DeclareMathOperator{\np}{\typesetoperator{NP}}
\DeclareMathOperator{\fpt}{\typesetoperator{FPT}}
\DeclareMathOperator{\xp}{\typesetoperator{XP}}

\newtheorem{theorem}{Theorem}
\newtheorem{lemma}[theorem]{Lemma}
\newtheorem{corollary}[theorem]{Corollary}
\newdefinition{definition}[theorem]{Definition}

\begin{document}

\begin{frontmatter}

\title{Directed branch-width: A directed analogue of tree-width.}

\author{Benjamin Merlin Bumpus\corref{cor1}\fnref{fn1}}\ead{benjamin.merlin.bumpus@gmail.com}\ead[url]{https://orcid.org/0000-0002-8686-2319}
\author{Kitty Meeks\corref{cor1}\fnref{fn2}}\ead{kitty.meeks@glasgow.ac.uk}\ead[url]{https://orcid.org/0000-0001-5299-3073}
\author{William Pettersson\corref{cor1}\fnref{fn3}}\ead{william.pettersson@glasgow.ac.uk}\ead[url]{https://orcid.org/0000-0003-0040-2088}

\address{(1) University of Florida, Computer \& Information Science \& Engineering, Florida, USA. \\
(2,3) School of Computing Science, University of Glasgow, UK}

\fntext[fn1]{Supported by EPSRC Doctoral Training Account EP/N509668/1. This author also received funding from the DARPA ASKEM and Young Faculty Award programs through grants HR00112220038
and W911NF2110323.}
\fntext[fn2]{Supported by a Royal Society of Edinburgh Personal Research Fellowship, funded by the Scottish Government. For the purpose of open access, the author has applied a Creative Commons Attribution (CC BY) licence to any Author Accepted Manuscript version arising.}
\fntext[fn3]{Supported by EPSRC grant EP/P028306/1.}

\cortext[cor1]{Corresponding author}

\begin{abstract}
Gurski and Wanke showed that a graph class $\mathcal{C}$ has bounded tree-width if and only if its associated class of directed line graphs has bounded clique-width. Inevitably -- asking whether this relationship lifts to directed graphs -- we introduce a new digraph width measure: we obtain it by investigating digraphs whose directed line graphs have bounded cliquewidth. Thus, to generalize Gurski and Wanke's aforementioned result, we introduce a natural generalization of branch-width to digraphs and we name it accordingly. 

Directed branch-width is a \textit{genuinely directed} width-measure insofar as it cannot be used to bound the value of the underlying undirected tree-width. Despite this, the two measures are still closely related: the directed branch-width of a digraph $D$ can differ from the branch-width of its underlying undirected graph only at sources and sinks. This relationship allows us to extend a range of algorithmic results from directed graphs with bounded underlying treewidth to the strictly larger class of digraphs having bounded directed branch-width. 
\end{abstract} 


\end{frontmatter}



\newpage

\section{Introduction}\label{sec:Intro}
One of the most important graph invariants is tree-width, which roughly 
measures how `tree-like' the global connectivity of any given graph is. Tree-width has many applications in parameterized complexity~\cite{cygan2015parameterized, flum2006parameterized, courcelle1990monadic} and
has played a key part in the proof of the celebrated Robertson-Seymour graph minor
theorem~\cite{RobertsonXX}.

Tree-width is also useful when graphs are directed: if we wish to solve a problem on a class of digraphs, we can hope for the tree-width of the underlying graph (i.e. the graph obtained by ignoring edge directions) to be bounded so that we may make use of its vast, accompanying algorithmic toolbox. But consider for example the Hamilton Path problem. It is $\np$-complete on undirected graphs but it is polynomial-time solvable on directed acyclic graphs (a class with unbounded underlying tree-width). Thus one observes that certain orientations of edges can actually make some problems tractable even on classes of digraphs with unbounded underlying tree-width. This motivates the search for a truly directed analogue of tree-width that can make such distinctions.

It is, however, not at all obvious how to obtain such an analogue. Indeed this has been an active research topic since the $1990$s and it has stimulated the research community to define numerous directed analogues of tree-width~\cite{kreutzer2018}. All of these digraph width measures (which we shall call `\emph{tree-width-inspired}' measures\footnote{Following the terminology in~\cite{kreutzer2018}, we do \textbf{not} include tree-width itself in this list.}, following Kreutzer and Kwon's terminology~\cite{kreutzer2018}) were defined by generalizing a characterization of tree-width in terms of cops-robber games; they include \emph{directed tree-width}\cite{dtw}, \emph{\texttt{DAG}-width}\cite{berwanger2012dag}, \emph{Kelly-width}\cite{hunter2008digraph} and \emph{\texttt{D}-width}\cite{safari2005d}. Unfortunately, despite being very useful in some cases, in general all tree-width-inspired measures face considerable algorithmic shortcomings. For example, they are all bounded on the class of \texttt{DAGs}, a class for which some natural problems remain $\np$-complete (e.g. directed Max-Cut~\cite{lampis2008algorithmic}). Indeed, it is known~\cite{ganian2010arethere} that, unless $\np \subseteq \mathbf{P}\setminus \mathbf{poly}$, certain algorithmic shortcomings are unavoidable by any digraph width measure which shares specific properties with any tree-width inspired measure (in particular there cannot be a Courcelle-like theorem parameterized by any one of these digraph width measures).

Another notable class of digraph width measures is that of `\emph{rank-width-inspired}' measures~\cite{kreutzer2018}. Rank-width is a graph invariant which measures how easily a given graph can be recursively decomposed into small parts via low-rank edge-cuts. Whenever a class of graphs has bounded tree-width, then its rank-width is also bounded. However the converse may fail. For example the classes of all cliques or of all complete bipartite graphs have bounded rank-width, but unbounded tree-width. There are several directed analogues of rank-width, including \emph{clique-width} (defined on digraphs from the start~\cite{courcelle1993}), \texttt{NLC}-width~\cite{gurski2016directedNLC} and \emph{bi-cut-rank-width}~\cite{kante2011}. These measures are unbounded on \texttt{DAGs} and they admit an algorithmic meta-theorem for problems definable in a more restricted fragment of logic~\cite{courcelle2000}. However, this fragment of logic does not have enough expressive power to describe all the problems which are tractable on classes of bounded undirected tree-width. In fact the directed versions of the Hamilton Path and Max-Cut problems (both tractable on classes of bounded underlying tree-width) are $\mathbf{W[1]}$-hard when parameterized by any of the foregoing tree-width- or rank-width-inspired measures~\cite{lampis2008algorithmic, fomin2010intractability}.  

As we mentioned, known directed analogues of tree-width are obtained by generalizing the `cops-robber games' definition of tree-width (whereby tree-width is defined as a number of tokens needed by one of the players to win this pursuit-evasion game; see Section~\ref{sec:dbw_properties}). In contrast, our incipit has little to do with cops-robbers games: we start from a result of Gurski and Wanke~\cite{gurski2007} which states that graph classes of bounded tree-width are precisely those classes of graphs whose line graphs have bounded clique-width. 

\textit{Our contribution} is thus born by asking whether it is possible to lift Gurski and Wanke's aforementioned result from graphs to directed graphs. We answer this question in the affirmative: we introduce the new notion of \emph{directed branch-width} and show the following theorem which establishes an equivalence between digraph classes of of bounded directed branch-width and digraph classes whose associated class of \emph{directed line graphs}\footnote{A directed line graph of a digraph $G$ is the directed graph having $E(G)$ as its vertices and directed paths of length two as its edges.} have bounded directed clique-width. 

\begin{theorem}\label{thm:dbw_and_directed_line-graph}
A class $\mathcal{C}$ of digraphs without parallel edges has bounded directed branch-width if and only if the class directed line-graphs of $\mathcal{C}$ has bounded bi-cut-rank-width. Specifically, for any digraph $D$ without parallel edges, we have 
\[\bcrk(\vec{L}(D))/2 - 1 \: \leq \: \dbw(D) \: \leq \: 8(1 + 2^{\bcrk(\vec{L}(D))}).\]
\end{theorem}

Directed branch-width turns out to be a natural generalization of the well-known notion of branch-width~\cite{robertsonX} (a graph invariant linearly-equivalent to tree-width) which opens new doors for the study of digraph connectivity (q.v. Section~\ref{sec:conclusion}) and directed line-graphs. As it turns out, although we show that underlying branch-width cannot be bounded in terms of our new measure (q.v. Theorem~\ref{thm:dbw_not_tw_bounding}), the two notions are nevertheless closely related: roughly-speaking (q.v. Lemma~\ref{lemma:dbw_invariant_source_sink_identification} and Corollary~\ref{corollary:dbw_bw_number_of_sources_and_sinks}), we prove that the directed branch-width of a digraph $D$ can differ from the branch-width of its underlying undirected graph only at sources and sinks.

Given the aforementioned importance of with measures for algorithmic applications, we would be remiss were we not to explore the algorithmic applications of directed branch-width. Thus, by leveraging connections between our new parameter and underlying tree-width, we show that the Hamilton Path problem and the directed Max-Cut problem are in $\fpt$ parameterized by directed branch-width (Corollary~\ref{corollary:fast_algs}). While this only generalizes the tractability with respect to underlying treewidth in a specific and limited way, it is nonetheless interesting given the fact that both of these problems are $\mathbf{W[1]}$-hard when parameterized by any tree-width-inspired or rank-width-inspired measure~\cite{lampis2008algorithmic, fomin2010intractability}. More generally we show (Corollary~\ref{corollary:our_courcelle}) that there exists an $\fpt$-time algorithm parameterized by directed branch-width for the model-checking problem on a restricted variant of the monadic second order logic of graphs (roughly this variant consists of those formulas whose corresponding model-checking problem is invariant under certain kinds of vertex-identification rules). Our results partially resolve some of the aforementioned questions about finding an `algorithmically useful' tree-width analogue in the sense that we provide a single width-measure which can be used to describe classes of digraphs for which many problems become tractable. However, since our algorithmic results lean heavily on the tractability of problems on classes of bounded underlying treewidth via standard techniques, all material relating to these algorithmic considerations is found in the Appendix. Given the graph-theoretic focus of the present paper, we leave it as further work to search for algorithmic results that make more native use of directed branch-width.

\section{Background}\label{sec:notation}
Throughout, for any natural number $n$, let $[n]$ denote the set $\{1, \ldots, n\}$. We denote by $xy$ the undirected edge joining two vertices $x$ and $y$ and we denote the directed edge starting at $x$ and ending at $y$ by $\dir{xy}$. For any vertex $x$ in a digraph $D$, we shall write $N_D^+(x)$ and $N_D^-(x)$ to mean the sets $\{y \in V(D): \dir{xy} \in E(D)\}$ and $\{w \in V(D): \dir{wx} \in E(D)\}$ respectively of \emph{out-neighbors} and \emph{in-neighbors} of $x$ in $D$. Furthermore, we shall denote by $N_D(x)$ the set $N_D^+(x) \cup N_D^-(x)$ of \emph{neighbors} of $x$. Note that, if the digraph $D$ is clear from context, then we shall drop the subscripts and simply write $N(x)$, $N^+(x)$ and $N^-(x)$. Given a directed graph $D$, we denote by $u(D)$ the \emph{underlying undirected graph} of $D$ which is obtained by making all of the edges of $D$ undirected. We say that a subset $A$ of vertices is \emph{complete} to another vertex-subset $B$ (disjoint from $A$) if every vertex of $A$ is adjacent to every vertex of $B$. 

The \emph{line-graph} of a graph $G$ was introduced by Whitney~\cite{whitney1992congruent} and is the graph defined as $(E(G), \{ef : e, f \in E(G) \text{ s.t. } e \cap f \neq \emptyset\})$. We introduce the concept of a \emph{directed line-graph} as a directed analogue of Whitney's definition. Just as the line-graph of an undirected graph encodes which pairs of edges can occur in succession in a path, the directed line-graph of a digraph $D$ encodes which edge-pairs can occur in succession in a directed path. 

\begin{definition}\label{def:dir_line_gr}
The \emph{directed line-graph} $\vec{L}(D)$ of a digraph $D$ is the digraph  
\[\vec{L}(D) := \bigl ( E(D), \bigl \{ \dir{ef}: \exists \{w,x,y\} \subseteq V(D) \text{ such that } e = \dir{wx} \text{ and } f = \dir{xy} \bigr\} \bigr ).\]
\end{definition}

Now we define the directed vertex and edge separators corresponding to edge and vertex-partitions respectively.

\begin{definition}\label{def:vert-sep}
Let $D$ be a directed graph, let $(V(D)\setminus A,A)$ be a partition of the vertices of $D$ and let $(E(D)\setminus B,B)$ be a partition of the edges of $D$. We call the sets 
\begin{align*}
S_A^E &:= \{\dir{xy} \in E(D): x \in V(D)\setminus A \text{ and } y \in A\} \text{ and}\\
S_B^V &:= \{y \in V(D): \exists x, z \in V(D) \text{ with } \dir{xy} \in E(D)\setminus B \text{ and } \dir{yz} \in B\} 
\end{align*}
respectively the \emph{edge separator} and \emph{vertex separator} corresponding to the partitions $(V(D)\setminus A,A)$ and~$(E(D)\setminus B,B)$. The \emph{directed order of an edge (or vertex) partition} is the number of elements contained in the vertex (or edge) separator associated with that partition (for example the order of $(E(D)\setminus B,B)$ is $|S_B^V|$). 
\end{definition}

Every edge of a tree $T$ partitions the set \emph{$\ell(T)$ of the leaves of $T$} into two sets. Given an edge $xy$ in a tree $T$ and letting $Y$ be the set of all leaves found in the connected component of $T - xy$ which contains the vertex $y$, we call the partition $\{ \ell(T) \setminus Y, Y \}$ the \emph{leaf-partition} of $T$ corresponding to $xy$. 

For any graph-theoretic notation not defined here, we refer the reader to~\cite{Diestel2010GraphTheory}.

\subsection{Tree-width-inspired measures.}
Tree-width is a measure of global connectivity and `tree-likeness' which was introduced independently by many authors~\cite{Bertele1972NonserialProgramming, Halin1976S-functionsGraphs, RobertsonII}. At an intuitive level, it can be thought of as a measure of how how far a graph is from being a tree; for example edge-less graphs have tree-width $0$, forests with at least one edge have tree-width $1$ and, for $n > 1$, $n$-vertex cliques have tree-width $n-1$.

\begin{definition}\label{def:treedecompositionDefinition}
Let $G$ be a graph, let $T$ be a tree and let $(V_t)_{t \in T}$ be a sequence of vertex subsets (called bags) of $G$ indexed by the nodes of $T$. The pair $(T, (V_t)_{t \in T})$ is a \emph{tree decomposition} of $G$ if it satisfies the following axioms:
\begin{enumerate}
\item every vertex and every edge of $G$ is contained in at least one bag of $(T, (V_t)_{t \in T})$
\item for any vertex $x$ in $G$, the induced subgraph $T[\{t \in T : x \in V_t\}]$ is connected.
\end{enumerate}
\end{definition}

The \emph{width} of a tree decomposition $(T, (V_t)_{t \in T})$ of a graph $G$ is defined as $\max_{t \in T} |V_t| - 1$. The \emph{tree-width} $\tw(G)$ of $G$ is the minimum possible width of any tree decomposition of $G$. 

There are several equivalent definitions of tree-width. One unified way of defining both undirected tree-width and all of the tree-width-inspired measures for digraphs~\cite{kreutzer2018} is via a family of pursuit-evasion games called \emph{cops-robber games} which are played on either directed or undirected graphs. These games are played on any (directed) graph with two players: the \emph{Sheriff} and the \emph{Villain}. The Sheriff controls a set of cops which each occupy a single vertex and the Villain controls a single robber which also occupies a single vertex. The Sheriff wins if a cop is ever placed on a vertex occupied by the robber and looses otherwise. Different cops-robber games can be defined by varying how the cops and the robbers can be moved and by the information available to each player. Depending on the game variant, undirected tree-width and each tree-width-inspired width measure can be defined as the number of cops needed to to guarantee a win for the Sheriff~\cite{cygan2015parameterized, kreutzer2018}. For a formal description of these games, see Section \ref{subsec:dbw_comparison_to_other_measures}.

\subsection{Layouts: branch-width and rank-width.}
A notion closely related to tree-width, which we use throughout this paper, is that of \emph{branch-width}. 

\begin{definition}[\cite{robertsonX}]
A \emph{branch decomposition} of a graph $G$ is a pair $(T,\tau)$ where $T$ is a tree of maximum degree three and $\tau$ is a bijection from the leaves of $T$ to $E(G)$. The \emph{order of an edge} $e$ of $T$ is the number of vertices $v$ of $G$ such that there are leaves $t_1$, $t_2$ in $T$ in different components of $T-e$, with $\tau(t_1)$, $\tau(t_2)$ both incident with $v$. The \emph{width} of $(T,\tau)$ is the maximum order of the edges of $T$. The \emph{branch-width} of $G$ (written $\bw(G)$) is the minimum possible width of any branch-decomposition of $G$.
\end{definition}

It is known that a class of undirected graphs has bounded tree-width if and only if it has bounded branch-width.

\begin{theorem}[\cite{robertsonX}]\label{thm:bounded_tw_bounded_bw}
For any graph $G$, $\bw(G) \leq \tw(G) + 1 \leq \max \{1, 3\bw(G)/2\}$.
\end{theorem} 

We note that, if a graph has fewer than two edges, then it has branch-width zero; furthermore, all stars have branch-width at most $1$ and all forests have branch-width at most $2$ (this follows by Theorem \ref{thm:bounded_tw_bounded_bw} since forests have tree-width at most $1$; the bound is tight since the three-edge path has branch-width $2$).

A branch-decomposition is a special case of the more general concept of a \emph{layout of a symmetric function} which can be applied even to structures which are not graphs. Given sets $A$ and $B$, a function $f: 2^A \to B$ is \emph{symmetric} if, for any $X \subseteq A$ we have $f(X) = f(A \setminus X)$.

\begin{definition}
Let $U$ be a finite set and let $f: 2^U \rightarrow \mathbb{Z}$ be a symmetric function. We say that the pair $(T,\beta)$ is a \emph{layout of $f$ on $U$} if $T$ is a tree of maximum degree three and $\beta$ is a bijection from the leaves of $T$ to the elements of $U$. 

Let $xy$ be an edge in $T$ and let $\{ \ell(T) \setminus Y, Y \}$ be the leaf-partition associated with $xy$. We define the \emph{order of the edge} $xy$ (denoted $||xy||_f$) to be the value of $f(\beta(Y))$. The \emph{width} of $(T,\tau)$ is the maximum order of all edges of $T$. The \emph{layout-$f$-width of $U$} is defined as the minimum possible width of any layout of $f$ on $U$.
\end{definition}

Using the concept of a layout, we can now give an alternative definition of the branch-width of a graph $G$: the branch-width of $G$ is the layout-$f$-width of $E(G)$ where $f$ maps any edge subset $X$ of $G$ to the number of vertices incident with both an edge in $E(G) \setminus X$ and an edge in $X$.

Whereas branch-width can be thought of as a global measure of vertex-connectivity, \emph{rank-width} (introduced in~\cite{oum2006approximating}) is intuitively a measure of global neighborhood similarity. All classes of bounded branch-width have also have bounded rank-width, but the converse is not true: for example the rank-width of any clique is $1$. Given a matrix $M$ over $\mathbb{GF}(2)$ whose rows and columns are indexed by some set $X$, we denote by $\rk(M)$ the rank of $M$ and, for $Y \subseteq X$, we denote by $\rk(M[X\setminus Y, Y])$ the rank of the sub-matrix of $M$ given only by the rows indexed by elements of $X \setminus Y$ and the columns indexed by the elements of $Y$. If $M$ is the adjacency matrix of a graph $G$ and $X$ is a vertex subset of $G$, then $\rk(M[V(G)\setminus X, X])$ describes the number of different kinds of edge-interactions between vertices in $V(G) \setminus X$ and those in $X$; Oum and Seymour used layouts to turn this local notion into a global one by defining rank-width as follows~\cite{oum2006approximating}. 

\begin{definition}[\cite{oum2006approximating}]
Let $M$ be the adjacency matrix over $\mathbb{GF}(2)$ of a graph $G$. The \emph{rank-width} $\rw(G)$ of $G$ is the layout-$f$-width of $V(G)$ where we define $f$ as the mapping $f: X \mapsto \rk(M[V(G)\setminus X, X])$. 
\end{definition}

Kant\'{e} and Rao generalized rank-width to digraphs by introducing the concept of a \emph{bi-cut-rank decomposition}~\cite{kante2011}. They did so by considering layouts of a function which takes into account the two possible directions in which edges can point in the directed edge separator associated with a vertex partition.

\begin{definition}[\cite{kante2011}]\label{def:bi-cut-rank-width}
Given a digraph $D$ and its adjacency matrix $M$ over $\mathbb{GF}(2)$, let $f(M[X])$ be the function defined as $f(M[X]) = \rk(M[V(D)\setminus X, X]) + \rk(M[X, V(D)\setminus X])$. The \emph{bi-cut-rank-width} of a digraph $D$ (denoted $\bcrk(D)$) is the layout-$f$-width of $V(D)$. 
\end{definition} 

For completeness we note that rank-width is closely related to a width-measure called \emph{clique-width} (defined in~\cite{courcelle1993}).

\begin{theorem}[\cite{gurski2016directedNLC},~\cite{oum2006approximating}]\label{thm:clique_width_equivalence}
A class of undirected graphs has bounded clique-width if and only if has bounded rank-width.
\end{theorem}

Gurski and Wanke showed that classes of bounded tree-width can be characterized via the rank-width of their line graphs~\cite{gurski2007}. (We note that this result was originally stated in terms of clique-width; by Theorem \ref{thm:clique_width_equivalence}, the following formulation is equivalent.) 

\begin{theorem}[\cite{gurski2007},\cite{oum2008rank}]\label{thm:gurski_wanke}
A class of undirected graphs has bounded tree-width if and only if the class of its line graphs has bounded rank-width.
\end{theorem}

Theorem \ref{thm:gurski_wanke} is of particular importance since it can be seen as a definition of classes of bounded tree-width. In Section \ref{sec:line_graphs} we will prove a directed analogue of Theorem \ref{thm:gurski_wanke}. 


\section{Proof of Theorem~\ref{thm:dbw_and_directed_line-graph}: Directed line graphs of bounded rank-width.}\label{sec:line_graphs}
In this section we introduce our new directed width measure, directed branch-width, and use it to generalize Theorem \ref{thm:gurski_wanke} to digraphs. Specifically we will prove Theorem~\ref{thm:dbw_and_directed_line-graph} which, stated in terms of graph classes, says: a class $\mathcal{C}$ of digraphs has bounded directed branch-width if and only if the class $\vec{L}(\mathcal{C})$ of directed line graphs of $\mathcal{C}$ has bounded bi-cut-rank-width. Note that all directed analogues of rank-width~\cite{kante2011} are bounded on the same graph classes as bi-cut-rank-width~\cite{kreutzer2018}, so this result can also be stated in terms of any rank-width-inspired measure.

\begin{figure}[h]
\centering
\begin {tikzpicture}[-latex , auto , node distance =1.2 cm and 1.2 cm , on grid ,
semithick ,
state/.style ={ circle ,top color =softGray , bottom color = creamWhite ,
draw, softBlack , text=softBlack , minimum width =0.2 cm}]
\node[state] (A) {$a$};
\node[state] (B) [right=of A] {$b$};
\node[state] (C) [right=of B] {$c$};
\node[state] (D) [below =of A] {$d$};
\node[state] (E) [below =of B, top color = brick!50] {$e$};
\node[state] (F) [below =of C, top color = brick!50] {$f$};
\node[state] (G) [below =of D] {$g$};
\node[state] (H) [below =of E] {$h$};
\node[state] (I) [below =of F] {$i$};

\path (A) edge (B);
\path (B) edge (C);
\path (D) edge (A);
\path (B) edge (E);
\path (F) edge (C);
\path (E) edge (D);
\path (F) edge [line width=0.50mm, color = brick, dotted] (E);
\path (G) edge (D);
\path (E) edge [line width=0.50mm, color = brick, dotted] (H);
\path (I) edge [line width=0.50mm, color = brick, dotted] (F);
\path (G) edge (H);
\path (I) edge [line width=0.50mm, color = brick, dotted] (H);

\node[state] (DE) [right =3.5cm of C ] {$\dir{ed}$};
\node[state] (AD) [right =of DE] {$\dir{da}$};
\node[state] (AB) [right =of AD] {$\dir{ab}$};
\node[state] (BE) [right =of AB] {$\dir{be}$};

\node[state] (t1) [below right =of DE] {};
\node[state] (t2) [below right =of AB] {};
\node[state] (t3) [below =of AB] {};
\node[state] (t4) [below =of t3] {};
\node[state] (t5) [left =of t4] {};
\node[state] (t6) [right =of t4] {};
\node[state] (t7) [left =of t5] {};
\node[state] (t8) [below =of t5] {};
\node[state] (t9) [below left =of t8] {};
\node[state] (t10) [right =of t8] {};

\node[state] (DG) [above =of t7] {$\dir{gd}$};
\node[state] (GH) [below =of t7] {$\dir{gh}$};

\node[state] (BC) [right =of t6] {$\dir{bc}$};
\node[state] (CF) [below =of t6] {$\dir{fc}$};

\node[state] (EH) [left =of t9, top color = brick!60] {$\dir{eh}$};
\node[state] (HI) [right =of t9, top color = brick!60] {$\dir{ih}$};
\node[state] (FI) [below =of t10, top color = brick!60] {$\dir{if}$};
\node[state] (EF) [below =of CF, top color = brick!60] {$\dir{fe}$};

\path (DE) edge [-] (t1);
\path (AD) edge [-] (t1);
\path (AB) edge [-] (t2);
\path (BE) edge [-] (t2);
\path (t1) edge [-] (t3);
\path (t2) edge [-] (t3);
\path (t3) edge [-] (t4);
\path (t4) edge [-] (t5);
\path (t4) edge [-] (t6);
\path (t5) edge [-] (t7);
\path (t5) edge [-, color = brick, line width=0.80mm] node[right =0.15 cm, color = softBlack] {$\xi$} (t8);
\path (t8) edge [-] (t9);
\path (t8) edge [-] (t10);
\path (DG) edge [-] (t7);
\path (GH) edge [-] (t7);
\path (BC) edge [-] (t6);
\path (CF) edge [-] (t6);
\path (EH) edge [-] (t9);
\path (HI) edge [-] (t9);
\path (FI) edge [-] (t10);
\path (EF) edge [-] (t10);
\end{tikzpicture}
\caption{
    An orientation $D$ of a $(3 \times 3)$-grid (left) and a directed branch decompositions of this grid (right).  Letting $X = \{ \protect\dir{eh},\protect\dir{ih},\protect\dir{if},\protect\dir{fe} \}$, the edge $\xi$ is associated with the edge partitions $(E(G) \setminus X, X)$ and $(X, E(G) \setminus X)$. These partitions are themselves respectively associated with the directed vertex separators $\{ e \}$ and $\{ e, f \}$.
}\label{fig:decomp_example}
\end{figure}
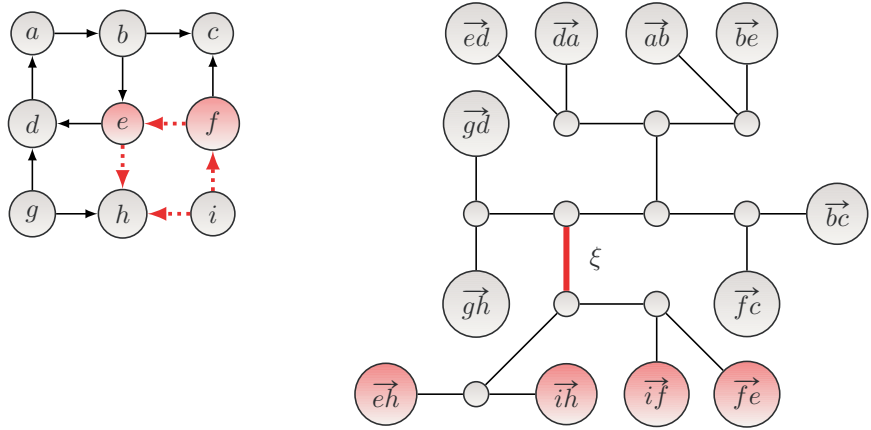

The definition of directed branch-width relies on the notion of directed vertex-separator. Recall from Definition \ref{def:vert-sep} that, for any edge-subset $X$ of a digraph $D$, $S_X^V$ and $S_{E(D)\setminus X}^V$ denote the directed vertex-separators corresponding to $(E(D)\setminus X, X)$ and $(X, E(D)\setminus X)$ respectively.

\begin{definition}[Directed branch-width]
For any digraph $D$, let $f_D$ be the function $f_D: 2^{E(D)} \to \mathbb{N}$ defined as $f_D(X) = |S_X^V \cup S_{E(D)\setminus X}^V|$.
We call any layout of $f_D$ on $E(D)$ a \emph{directed branch decomposition} of $D$.
The \emph{directed branch-width} of $D$, denoted $\dbw(D)$, is the layout-$f_D$-width of $E(D)$. (See Figure~\ref{fig:decomp_example} for an example.) 
\end{definition}

In view of proving Theorem~\ref{thm:dbw_and_directed_line-graph}, we shall first show (Lemma \ref{lemma:dbw_to_brw}) that if a class $\mathcal{C}$ has bounded directed branch-width then $\vec{L}(\mathcal{C})$ has bounded bi-cut-rank-width. Then we will obtain the converse of this statement in Lemma \ref{lemma:bcrk_dbw}. These results rely on an auxiliary notion of consistency between a labeling function and a vertex-partition. 

\begin{definition}\label{def:lambda-consistent}
Let $S$ be a set, $(A,B)$ be a vertex partition in a digraph $D$ and let $\lambda_A: A\rightarrow S$ and $\lambda_B: B \rightarrow S$ be functions. We say that \emph{$(A,B)$ is $(\lambda_A, \lambda_B)$-consistent} if 
\begin{itemize}
\item given vertices $a_1$, $a_2$ in $A$ with $\lambda_A(a_1) = \lambda_A(a_2)$ and a vertex $b$ in $B$ we have that $\dir{a_1b} \in E(D)$ if and only if $\dir{a_2b} \in E(D)$ and
\item given vertices $b_1$, $b_2$ in $B$ with $\lambda_B(b_1) = \lambda_B(b_2)$ and a vertex $a$ in $A$ we have that $\dir{ab_1} \in E(D)$ if and only if $\dir{ab_2} \in E(D)$.
\end{itemize}

\end{definition}

For clarity we emphasize that the $(\lambda_A,\lambda_B)$-consistency of a partition $(A,B)$ only concerns the edges that emanate from $A$ and enter $B$. Thus the fact that a vertex-partition $(A,B)$ of a digraph $D$ is $(\lambda_A,\lambda_B)$-consistent does not necessarily imply that $(B,A)$ is also consistent with $(\lambda_B,\lambda_A)$.

Oum showed that, if $L(G)$ is the undirected line-graph of an undirected graph $G$, then $\rw(L(G)) \leq \bw(G)$~\cite{oum2008rank}. We will now show a directed analogue of this result with a weaker bound: the bi-cut-rank-width of a directed line graph of some digraph $D$ is bounded by a linear function of the directed branch-width of $D$. 

\begin{lemma}\label{lemma:dbw_to_brw}
Let $D$ be a digraph; if $\dbw(D) \leq k$, then $\bcrk(\vec{L}(D)) \leq 2(k+1)$.
\end{lemma}
\begin{proof}
Throughout, let $M$ be the adjacency matrix of $\vec{L}(D)$ over $\mathbb{GF}(2)$ whose rows and columns are indexed by $E(D)$. Let $(T, \beta)$ be a directed branch decomposition of $D$ of width at most $k$.

Notice that since $E(D) = V(\vec{L}(D))$ and since the leaves of $T$ correspond bijectively (via $\beta$) to the elements of $E(D)$, it follows that they are also in bijective correspondence with $V(\vec{L}(D))$ (again, witnessed by $\beta$). In particular, this means that $(T, \beta)$ can be also seen as a bi-cut-rank-decomposition of $\vec{L}(D)$. Thus it suffices to show that, if $(E(D) \setminus X, X)$ is any edge partition of $D$ corresponding to an edge of $T$, then the value of $\rk(M[E(D) \setminus X, X])$ is at most $k + 1$ since then this would imply that, when viewed as a bi-cut-rank decomposition, the width of $(T, \beta)$ is at most $2(k+1)$.

Now fix a vertex partition $(E(D) \setminus X, X )$ of $\vec{L}(D)$ (recall that edge partitions of $D$ are vertex partitions of $\vec{L}(D)$) and suppose that it has directed order at most $k$ in $D$. Furthermore, let $\lambda_1 : E(D) \setminus X \rightarrow S$ and $\lambda_2 : X \rightarrow S$ be labeling functions. Notice that, if $(E(D) \setminus X, X )$ is $(\lambda_1, \lambda_2)$-consistent, then it must be that any two rows of $M[V(\vec{L}(D))\setminus X,X]$ indexed by vertices of $V(\vec{L}(D))\setminus X$ which have the same label under $\lambda_1$ are identical and hence linearly dependent. In particular this implies that if a vertex partition $(E(D) \setminus X, X )$ of $\vec{L}(D)$ is $(\lambda_1, \lambda_2)$-consistent, then $\rk(M[E(D) \setminus X, X ]) \leq |S|$ since there are at most $|S|$ possible labels. 

It therefore remains only to construct labeling functions $\lambda_1$ and $\lambda_2$ mapping the edges of $D$ to a set with at most $k + 1$ elements and show the $(\lambda_1, \lambda_2)$-consistency of $(E(D) \setminus X, X)$.

Let $\{ s_1, \ldots, s_r\}$ be the vertex separator of the partition $(E(D)\setminus X, X)$ (note that, since $(T, \beta)$ has width at most $k$, we have $r \leq k$). Define $\lambda_1 : E(D) \setminus X \rightarrow [r] \cup \{0\}$ and $\lambda_2 : X \rightarrow  [r] \cup \{0\}$ as: 
\begin{align*}
\lambda_1(\dir{xy}) &= 
            \begin{cases} 
                i     \text{ if } y = s_i \\
                0     \text{ otherwise}.
            \end{cases} \\
\lambda_2(\dir{xy}) &= 
            \begin{cases} 
                i     \text{ if } x = s_i  \\
                0     \text{ otherwise}.
            \end{cases}
\end{align*}
Note that an edge $\dir{xy}$ will have label zero if it is in $E(D) \setminus X$ and does not point towards an element of $\{ s_1, \ldots, s_r\}$, or if it is an element of $X$ and it does not emanate from an element of $\{ s_1, \ldots, s_r\}$.

We claim that $(E(D) \setminus X, X)$ is a $(\lambda_1,\lambda_2)$-consistent vertex-partition in $\vec{L}(D)$ (again recall that edge-partitions in $D$ are vertex-partitions in $\vec{L}(D)$). To see this, let $e_1$ and $e_2$ be two elements of $E(D)\setminus X$ with $\lambda_1(e_1) = \lambda_1(e_2) = i$. If $i = 0$, then neither $e_1$ nor $e_2$ points towards a vertex in $\{s_1, \ldots, s_r\}$. Thus, for all $f$ in $X$, neither $\dir{e_1f}$ nor $\dir{e_2f}$ is an edge in $\vec{L}(D)$. 
Otherwise, if $i \neq 0$, then $e_1$ and $e_2$ both point towards the vertex $s_i$. This means that, for any $f \in X$ and $j \in [2]$, there will be an edge $\dir{e_j f}$ in $\vec{L}(D)$ if and only if $f$ points away from $s_i$. Thus $\dir{e_1 f} \in E(\vec{L}(D))$ if and only if $\dir{e_2 f} \in E(\vec{L}(D))$. A symmetric argument also shows that for $f_1, f_2 \in X$ and $e \in E(D) \setminus X$, if $\lambda_2(f_1) = \lambda_2(f_2)$, then $\dir{e f_1} \in E(\vec{L}(D))$ if and only if $\dir{e f_2 } \in E(\vec{L}(D))$. The partition $(E(D) \setminus X, X)$ is therefore $(\lambda_1,\lambda_2)$-consistent and the result follows.
\end{proof}

Our next goal is to prove the converse of Lemma \ref{lemma:dbw_to_brw}: given a digraph $D$ such that $\vec{L}(D)$ has bi-cut-rank-width at most $k$, we shall deduce a bound on the directed branch-width of $D$. We obtain this result by extending techniques introduced by Gurski and Wanke in~\cite{gurski2007} to the setting of digraphs. Once again we shall be using the concept of consistency with respect to a labeling function as an intermediate step in our results. To do so, we first need an auxiliary result giving a bound on the number of labels needed in order to find labeling functions with which a vertex partition of rank $k$ is consistent.

\begin{lemma}\label{lemma:rank_to_lambda_consistent}
Let $(V(D) \setminus X, X)$ be a vertex partition of a digraph $D$ and let $M$ be the adjacency matrix of $D$ over $\mathbb{GF}(2)$. If $\rk(M[V(D) \setminus X, X]) \leq k$, then there exists a set $S$ with $|S| \leq 1 + 2^k$ and labeling functions $\lambda_1: V(D) \setminus X \to S$ and $\lambda_2: X \to S$ such that $(V(D)\setminus X, X)$ is $(\lambda_1, \lambda_2)$-consistent.
\end{lemma}
\begin{proof}
Let $A$ be the subset of all vertices in $V(D) \setminus X$ which have an outgoing edge to an element of $X$ and let $B$ be the set of all vertices in $X$ which have an incoming edge from an element of $V(D) \setminus X$. We define two equivalence relations $\sim_A$ and $\sim_B$ on $A$ and $B$ as: 
\begin{itemize}
\item for all $x$ and $y$ in $A$, $x \sim_A y$ if $N^+(x) \cap X = N^+(y) \cap X$,
\item for all $x$ and $y$ in $B$, $x \sim_B y$ if $N^-(x) \cap V(D) \setminus X = N^-(y) \cap V(D) \setminus X$.
\end{itemize}

Let $\{a_1, \ldots , a_p\}$ and $\{b_1, \ldots , b_q\}$ sets of representatives of the equivalence classes of $\sim_A$ and $\sim_B$ in $A$ and $B$ respectively. Now we define the labeling functions $\lambda_1: V(D) \setminus X \to [p] \cup \{ 0 \}$ and $\lambda_2: X \to [q] \cup \{ 0 \}$ as 
\begin{align*}
    \lambda_1(x) &= 
    \begin{cases}
        0 \text{ if } x \not \in A \\
        i \text{ if } x \text{ is in the equivalence class represented by } a_i \text{ under } \sim_A
    \end{cases}\\
    \lambda_2(y) &= 
    \begin{cases}
        0 \text{ if } y \not \in B \\
        j \text{ if } y \text{ is in the equivalence class represented by } b_j \text{ under } \sim_B .
    \end{cases}
\end{align*}
These definitions imply that two vertices in either $V(D) \setminus X$ or $X$ have the same label under either $\lambda_1$ or $\lambda_2$ if and only if they share exactly the same out-neighbors in $X$ or in-neighbors in $V(D) \setminus X$ respectively. Thus $(V(D) \setminus X, X)$ is $(\lambda_1, \lambda_2)$-consistent.

All that remains to be shown is that $p$ and $q$ are both at most $2^k$. We will argue only for $p$ since the argument is the same (replacing rows for columns and in-neighborhoods for out-neighborhoods) for the bound on $q$. To see that $p \leq 2^k$, note that two vertices in $A$ have the same out-neighborhood in $X$ if and only if the row-vectors that they index in $M[V(D) \setminus X, X]$ are identical. Since $M[V(D) \setminus X, X]$ has rank at most $k$, there are at most $2^k$ linear combinations with binary coefficients of the basis vectors of the row-space of $M[V(D) \setminus X, X]$. This concludes the proof since $p$ is at most the number of out-neighborhoods in $X$.
\end{proof}

Now, for some digraph $D$, suppose we have a partition $(V(\vec{L}(D)) \setminus X, X)$ of rank $k$. By Lemma \ref{lemma:rank_to_lambda_consistent} we know that there exist functions $\lambda_1, \lambda_2$ mapping vertices of $V(\vec{L}(D)) \setminus X$ and $X$ to at most $\ell$ labels (for $\ell \leq 1 + 2^k$). Our next step is to bound the order of the edge-partition of $D$ corresponding to $(V(\vec{L}(D)) \setminus X, X)$ by a function of $\ell$. The following lemma is a translation of part of the proof of~\cite[Theorem 9]{gurski2007} to the setting of digraphs.

\begin{lemma}\label{lemma:edge_to_vertex_seps}
Let $D$ be a digraph without parallel edges and $(E(D) \setminus X, X)$ an edge-partition of $D$. Let $\lambda_1 \colon E(D) \setminus X \rightarrow S$ and $\lambda_2 \colon X \rightarrow S$ be functions to some $k$-element set $S$. If $(E(D) \setminus X, X)$ is $(\lambda_1, \lambda_2)$-consistent in $\vec{L}(D)$, then $(E(D) \setminus X, X)$ has directed-order at most $4k$ in $D$. 
\end{lemma}
\begin{proof}
Let $\xi \colon S \to 2^S$ be the function mapping any label $a$ in $S$ to the set
\[\xi(a) := \{b \in S: \exists e_1 \in \lambda_1^{-1}(a) \text{ and } \exists e_2 \in \lambda_2^{-1}(b) \text{ with } \dir{e_1e_2} \in E(\vec{L}(D))\}\]
of all labels in $S$ associated by $\lambda_2$ to vertices of $\vec{L}(D)$ which are the endopints of a directed edge emanating from some vertex labeled $a$ by $\lambda_1$ (for clarity, recall that $V(\vec{L}(D)) = E(D)$).
For any label $a$ in $S$, denote by $D_a$ the subgraph of $D$ with $V(D_a) = V(D)$ and edge-set \[E(D_a) = \lambda_1^{-1}(a) \cup \bigcup_{b \in \xi(a)}\lambda_2^{-1}(b).\]
We claim that it will suffice to show that, for every $a \in S$, the edge-partition $(\lambda_1^{-1}(a), \:  E(D_a) \setminus \lambda_1^{-1}(a))$ has directed-order at most $4$ in $D_a$. To see this, recall that every element of $E(D) \setminus X$ is labeled by $\lambda_1$ and every element of $X$ is labeled by $\lambda_2$. This means that, for any element $q$ in the vertex-separator of $D$ associated with $(E(D) \setminus X, X)$, there is a label $a$ such that $q$ belongs to the separator associated with $(\lambda_1^{-1}(a), \:  E(D_a) \setminus \lambda_1^{-1}(a))$. Thus, since there are at most $k$ labels (i.e. $k$ choices for $a$), showing that $|(\lambda_1^{-1}(a), \:  E(D_a) \setminus \lambda_1^{-1}(a))| \leq 4$ for every $a \in S$ would also then imply that $|(E(D)\setminus X, X)| \leq 4k$.

Thus let $a$ be any label in $S$ and, for notational simplicity, denote by $(A, B)$ the partion $(\lambda_1^{-1}(a), \:  E(D_a) \setminus \lambda_1^{-1}(a))$. 
Recall that, since $A$ and $B$ are edge-sets in $D$, they are vertex-sets in $\vec{L}(D)$. Furthermore, since $(E(D) \setminus X, X)$ is $(\lambda_1, \lambda_2)$-consistent in $\vec{L}(D)$, it follows (by the definitions of consistency and of $\xi$) that $A$ is complete to $B$ in $\vec{L}(D)$. In particular, this implies that in $D_a$ every element of $A$ is incident with every element of $B$.

We will now show that the partition $(A,B)$ of $D_a$  has order at most $4$. If $A$ contains at most $2$ edges, then they have at most $4$ endpoints and the claim is trivially true. So suppose that $A$ has at least $3$ edges. 

Consider first the case in which $A$ has at least two non-incident edges $e = \dir{wx}$ and $f = \dir{yz}$. Since every edge in $B$ is incident with every edge in $A$, we can deduce in particular that every edge of $B$ must be incident with both an element of $\{ w,x \}$ and an element of $\{ y,z \}$. Thus, since $e$ and $f$ are not incident, no edge of $B$ has an endpoint outside of $\{w,x,y,z\}$. We may therefore conclude that $(E(A),E(B))$ has order at most $4$ in $D_a$.

Otherwise, if $A$ does not have at least two non-incident edges and since it has no parallel edges, then the elements of $A$ must form a star in $D_a$. Since every edge of $B$ is incident with every edge of $A$ and since $A$ has more than two edges, it must be that every edge of $B$ is incident with the center of the star formed by $A$. In fact, since we are assuming no parallel edges, this implies that every edge of $B$ is incident \emph{exclusively} with the center of this star. Thus $(A,B)$ has order most $1$.
\end{proof} 

We emphasize that the requirement in Lemma \ref{lemma:edge_to_vertex_seps} for $D$ to not have any parallel edges is in necessary. To see this, consider the case in which $A$ is a star with all edges pointing towards the center and $B$ is the digraph with $V(A) = V(B)$ obtained from $A$ by reversing the directions of all the edges in $A$. In this case the order of $(A,B)$ in the digraph $Q := (V(A), E(A) \cup E(B))$ (which has parallel edges) would be $|A|$ (which equals $|B|$).

We now use Lemmas \ref{lemma:rank_to_lambda_consistent} and \ref{lemma:edge_to_vertex_seps} to bound the directed branch-width of a digraph $D$ in terms of the bi-cut-rank-width of its directed line-graph.

\begin{lemma}\label{lemma:bcrk_dbw}
If $D$ is a digraph without parallel edges with $\bcrk(\vec{L}(D)) \leq k$, then $\dbw(D) \leq 8(1 + 2^k)$.
\end{lemma}
\begin{proof}
Let $(T, \beta)$ be a minimum-width bi-cut-rank-decomposition of $\vec{L}(D)$ and let $M$ be the adjacency matrix of $\vec{L}(D)$. Since $V(\vec{L}(D)) = E(D)$, we know that $(T, \beta)$ is also a directed branch-decomposition of $D$. Now let $e$ be any edge in $T$ and let $(E(D \setminus X), X)$ be the edge-partition of $D$ associated with this edge.

Since $\bcrk(\vec{L}(D)) \leq k$, we know that $M[E(D) \setminus X, X]$ has rank at most $k$. By Lemma \ref{lemma:rank_to_lambda_consistent} we know that there exist vertex-labeling functions $\lambda_1$ and $\lambda_2$ such that every element of $V(\vec{L}(D))$ (i.e. $E(D)$) is mapped to one of at most $1 + 2^k$ labels and such that $(E(D \setminus X), X)$ is $(\lambda_1, \lambda_2)$-consistent. By Lemma \ref{lemma:edge_to_vertex_seps} the $(\lambda_1, \lambda_2)$-consistency of $(E(D \setminus X), X)$ in $\vec{L}(D)$ implies that the directed order of $(E(D \setminus X), X)$ in $D$ is at most $4(1 + 2^k)$. The result follows since the order of $e$ in $T$ is the cardinality of the union of the directed separators corresponding to $(E(D \setminus X), X)$ and $(X, E(D \setminus X))$.
\end{proof}

\noindent This concludes the proof of Theorem \ref{thm:dbw_and_directed_line-graph} since it follows immediately from Lemmas~\ref{lemma:dbw_to_brw} and~\ref{lemma:bcrk_dbw}.


\section{Properties of Directed branch-width.}\label{sec:dbw_properties}
Having introduced directed branch-width in the previous section, here we study some of its properties. In Section \ref{subsec:dbw_rel_to_bw} we will show that classes of bounded directed branch-width need not have bounded underlying undirected branch-width. Despite this result, we will obtain some relationships between undirected and directed branch-width which will be useful for our algorithmic applications in Section \ref{sec:algorithms}. In Section \ref{subsec:minors} we study butterfly minors and directed topological minors (two directed analogues of the minor relation) and we will show that directed branch-width is closed under the first but not under the second. Finally in Section \ref{subsec:dbw_comparison_to_other_measures} we shall show that classes of digraphs that have bounded width with respect to any tree-width or rank-width-inspired measure need not have bounded directed branch-width.

\subsection{Relationship to undirected branch-width.}\label{subsec:dbw_rel_to_bw}
Here we compare the directed branch-width of digraph and the branch-width of its underlying undirected graph. We will prove that there exist digraph classes of bounded directed branch-width and unbounded underlying undirected branch-width. However, despite this result, we will find two ways of relating directed and undirected branch-width. First we will show that, for any digraph $D$, the difference between the branch-width of $u(D)$ and the directed branch-width of $D$ is at most the number of sources and in sinks in $D$. Second we will show that, given any digraph $D$, we can obtain, by modifying $D$ only at sources and sinks, a digraph $H$ such that $\dbw(D) = \bw(u(H))$. All of these results will be important for our algorithmic applications in Section \ref{sec:algorithms}. 

Towards proving our results here and those in Section \ref{subsec:minors} about butterfly and directed topological minors, we observe that directed branch-width is subgraph-closed. 

\begin{lemma}\label{lemma:dbw_subgraph_closed}
If $H$ is a subgraph of a digraph $D$, then $\dbw(H) \leq \dbw(G)$.
\end{lemma}
\begin{proof}
Let $(T, \beta)$ be a directed branch decomposition of $D$. Let $T'$ be the inclusion-wise minimal subtree of $T$ containing the leaves in the set $\{\ell \text{ leaf in } T: \beta(\ell) \in E(H)\}$. Notice that, letting $\beta'$ be the restriction of $\beta$ to $T'$, the pair $(T',\beta')$ is a directed branch decomposition of $H$. Now take any partition $(E(D) \setminus X, X)$ of $D$ induced by an edge which is contained in both $T$ and $T'$. Since its order when restricted to the edge set of $H$ (i.e. the partition $(E(H) \setminus X, E(H) \cap X)$) is at most that of $(E(D) \setminus X, X)$, the width of $(T', \beta')$ is at most that of $(T,\beta)$.
\end{proof}

Now we introduce some more notation which will be convenient for the next results. Let $(T,\beta)$ be a directed branch decomposition of width $k$ of a digraph $D$. By the definition of $(T, \beta)$, every edge $e$ of $T$ is associated with two edge-partitions $(E(D)\setminus Y, Y)$ and $(Y, E(D)\setminus Y)$ of order at most $k$. We associate to every \emph{internal edge} $e$ of $T$ three sets: $X_e$, $X_{e,Y}$, $X_{e,E(D)\setminus Y}$. The sets $X_{e,Y}$ and $X_{e,E(D)\setminus Y}$ are the \emph{directed separators} at $e$ associated with the partitions $(E(D)\setminus Y, Y)$ and  $(Y, E(D)\setminus Y)$ respectively; we denote the set $X_{e,Y} \cup X_{e,E(D)\setminus Y}$ as $X_e$ and we call it the \emph{bidirected separator} at $e$. We point out that, using with this notation, the width of a directed branch-decomposition $(T, \beta)$ is $\max_{e \in E(T)}|X_e|$. 

In contrast to undirected graphs, in digraphs it is possible to have vertices of high degree (e.g. sources or sinks) which never appear as internal vertices in directed paths. The following result shows that such vertices never appear in bidirected separators of directed branch decompositions. 

\begin{lemma}\label{lemma:isolated_elements}
Let $e$ and $x$ be respectively an edge and a vertex in a digraph $J$. 
\begin{enumerate}
\item If $x$ never appears as an internal vertex of a directed path in $J$, then $x$ will never appear in a bidirected separator of any directed branch-decomposition of $J$.
\item If $e$ never appears in a directed path with at least two edges in $J$, then $\dbw(J - e) = \dbw(J)$.
\end{enumerate}
\end{lemma}
\begin{proof}
Note that (1) follows immediately from the definition of directed branch-width; furthermore (1) implies (2) since the endpoints of $e$ satisfy the conditions of (1). 
\end{proof}

This result immediately implies that there is an infinite set of digraphs given by acyclic orientations of the square grids which has bounded directed branch-width. To see this, let $\Gamma$ be the \emph{set of all square grid-graphs}. Define $\vec{\Gamma}$ to be the class of digraphs obtained by orienting the edges of every graph $G$ in $\Gamma$ as follows: take a proper two-coloring of $G$ with colors white and black and then orient every edge towards its black endpoint.

\begin{corollary}\label{corollary:grid_dbw}
Every graph in $\vec{\Gamma}$ has directed branch-width $0$.
\end{corollary}
\begin{proof}
Every vertex of any element $D$ of $\vec{\Gamma}$ is either a source or a sink. It therefore follows that every bidirected separator of any directed branch-decomposition of $D$ is empty (by Lemma \ref{lemma:isolated_elements}) and hence that $\dbw(D) = 0$. 
\end{proof}

Now we show that the boundedness of directed branch-width does not imply boundedness of underlying undirected branch-width. 

\begin{theorem}\label{thm:dbw_not_tw_bounding}
There does not exist any function $f: \mathbb{N} \to \mathbb{N}$ such that, for every digraph $D$, $\bw(u(D)) \leq f(\dbw(D))$.
\end{theorem}
\begin{proof}
By Corollary \ref{corollary:grid_dbw}, there is an infinite set $\vec{\Gamma}$ of acyclic orientations of the square grids which has directed branch-width zero. However, the $(n\times n)$-grid has tree-width $n$~\cite{Diestel2010GraphTheory} and hence it has branch-width at least $2n/3$ (by Theorem \ref{thm:bounded_tw_bounded_bw}).
\end{proof}

In a similar vein to Lemma \ref{lemma:isolated_elements}, we can show show that directed branch-width is invariant under a form of identifications of sources or sinks which we define now.

\begin{definition}\label{def:terminating_minor}
Two vertices $x$ and $y$ of a digraph $D$ are \emph{source/sink-identifiable} if they are either both sources or both sinks. We say that a digraph $H$ is \emph{equivalent to  $D$ under source-sink identification} if it can be obtained from $D$ via a sequence of identifications of pairs of source/sink-identifiable vertices. We say that a graph $D'$ is a \emph{source-sink-split} of a digraph $D$ if every source and every sink in $D'$ has degree exactly $1$ and if $D$ is equivalent to $D'$ under source-sink identification.
\end{definition}

The next result shows that directed branch-width is invariant under source-sink identifications. Note that, for undirected graphs, where the natural analogue of a source or a sink is a pendant vertex, this does not hold. In fact, the branch-width of an undirected graph can be increased arbitrarily by repeated identification of pendant vertices.

\begin{lemma}\label{lemma:dbw_invariant_source_sink_identification}
Let $H$ and $D$ be two digraphs. If $H$ is equivalent to $D$ under source-sink identification, then $\dbw(H) = \dbw(D)$. 
\end{lemma}
\begin{proof}
Let $x$ and $y$ be two source/sink-identifiable vertices in $D$ and note that it is sufficient to consider the case in which $H$ is obtained from $D$ by identifying $x$ and $y$ into a vertex $\gamma$.

Let $(T,\beta)$ be a directed branch decomposition of $D$ and define the map $\omega: E(D) \to E(H)$ as 
\[\omega(uv) = \begin{cases}
                \dir{\gamma v} \text{ if }  u \in \{ x , y \} \\
                \dir{u \gamma} \text{ if }  v \in \{ x , y \} \\
                uv \text{ otherwise}
              \end{cases}\]
Clearly the map $\xi := \omega \circ \beta$ from the leaves of $T$ to $E(H)$ is surjective; however, it would fail to be injective if $x$ and $y$ have a common neighbor. In this case, we can simply remove some duplicate leaves (i.e. leaves mapped by $\xi$ to the same edge) in $T$ so that we can ensure $\xi$ is injective. We shall assume that we have done so to $(T, \xi)$ if required.

Since $x$ and $y$ are $st$-identifiable, they never appear as internal vertices of a directed path. Thus, by Lemma \ref{lemma:isolated_elements},  it follows that neither $x$ nor $y$ will ever appear in any bidirected separator of $(T,\beta)$. Furthermore, we know that $\gamma$ must also be either a source or a sink since identifying two sources or two sinks yields respectively a source or a sink. Thus $\gamma$ will also never appear in any bidirected separator of $(T,\xi)$ (again by Lemma \ref{lemma:isolated_elements}). In particular, letting $X$ be any edge-subset of $D$, this implies that the order of $(E(D)\setminus X, X)$ in $G$ is the same as that of $(E(H) \setminus \omega(X), \omega(X))$ in $H$. By the definition of $\xi$ this implies that the width of $(T, \beta)$ is the same as that of $(T, \xi)$.
\end{proof}

The following result characterizes bi-directed vertex-separators in a digraph $D$ in terms of vertex separators in $u(D)$ and the set of sources and sinks in $D$.

\begin{lemma}\label{lemma:dir_separtor_is_undir_separator_without_sources_or_sinks}
Let $(E(D) \setminus X, X)$ be an edge-partition of $D$ and let $S_1$ and $S_2$ be the directed separators associated with $(E(D) \setminus X, X)$ and $(X, E(D) \setminus X)$ respectively. Let $U$ be the set of all vertices incident both with an edge in $E(u(D)) \setminus X$ and with an edge in $X$ in the underlying undirected graph $u(D)$. If $S$ is the set of all sources and all sinks in $D$, then $S_1 \cup S_2 = U \setminus S$.
\end{lemma}
\begin{proof}
Consider any element $x$ in $S_1 \cup S_2$. By the definition of $S_1$ and $S_2$, there must be two edges $e \in E(D) \setminus X$ and $f \in X$ such that either $e = \dir{wx}$ and $f = \dir{xy}$ or such that $e = \dir{xw}$ and $f = \dir{yx}$. Either way, this implies that $x$ is neither a source nor a sink and that $x \in U$ (since $x$ is incident with both $e$ and $f$ in $u(D)$). Thus we have shown that $S_1 \cup S_2 \subseteq U \setminus S$.

Now consider any element $y \in U \setminus (S_1 \cup S_2)$. Since $y$ is in $U$, we know that there is at least one pair of edges $(g,h) \in (E(D) \setminus X) \times X$ incident with $y$. However, since $y$ is not in $S_1 \cup S_2$, both $g$ and $h$ either point away from $y$ or towards $y$. This means that $y$ is an element of $S$ (i.e. it is a source or a sink). Thus $S_1 \cup S_2 = U \setminus S$.
\end{proof}

Via Lemma \ref{lemma:dir_separtor_is_undir_separator_without_sources_or_sinks} we now find a relationship between directed branch-width of a digraph $D$ and the undirected branch-width of its underlying undirected graph which depends on the number of sources and sinks of $D$. 

\begin{corollary}\label{corollary:dbw_bw_number_of_sources_and_sinks}
Let $D$ be a digraph; if $S$ is the set of all sources and sinks in $D$, then $\bw(u(D)) - |S| \leq \dbw(D) \leq \bw(u(D)). $
\end{corollary}
\begin{proof}
Consider a directed branch-decomposition $(T, \beta)$ of $D$. Letting $\phi: E(D) \to E(u(D))$ be the bijection mapping every directed edge to its undirected counterpart, note that $(T, \phi \circ \beta)$ is an undirected branch decomposition of $u(D)$. Furthermore, observe that, in this way, every undirected branch decomposition can be obtained from some directed branch decomposition and vice versa. 

Let $\varepsilon$ be any edge in $T$ and suppose that it has order $k$ in $(T, \beta)$ and order $\ell$ in $(T, \phi \circ \beta)$. By Lemma \ref{lemma:dir_separtor_is_undir_separator_without_sources_or_sinks}, we know that $\ell - |S| \leq k \leq \ell$. Furthermore, this is true for any choice of $\varepsilon$. Thus, since $\phi$ establishes a bijective correspondence between the set of all directed branch decompositions of $D$ and that of all undirected branch decompositions of $u(D)$, it follows that $\bw(u(D)) - |S| \leq \dbw(D) \leq \bw(u(D))$. 
\end{proof}

Note that, since any bidirected orientation $D$ of an undirected graph $G$ has no sources or sinks, we can apply Corollary \ref{corollary:dbw_bw_number_of_sources_and_sinks} to deduce the following result. 
\begin{corollary}\label{corollary:dbw_bidirected_orientations}
If $D$ is the digraph obtained by replacing every edge $xy$ in an undirected graph $G$ with the edges $\dir{xy}$ and $\dir{yx}$, then $\dbw(D) = \bw(G)$.
\end{corollary}

We will now show that, given any digraph $D$, we can obtain, by modifying $D$ only at sources and sinks, a digraph $D'$ such that the directed branch-width of $D$ equals the underlying undirected branch-width of $D'$. 

\begin{lemma}\label{lemma:dbw_undir_lower_bound}
If $H$ is the source-sink-split of $D$, then $\bw(u(H)) = \dbw(D)$.
\end{lemma}
\begin{proof}
By Lemma \ref{lemma:dir_separtor_is_undir_separator_without_sources_or_sinks} we know that the vertex separator in $u(H)$ consists of the elements of the vertex separator it corresponds to in $H$ as well as possibly some sources and/or sinks. By the definition of $H$, every source or sink in $H$ has degree $1$. Thus, for any vertex $x \in V(H)$ (recall $V(H) = V(u(H))$), $x$ will appear as an internal vertex in a directed path in $H$ if and only if it appears as an internal vertex in a path in $u(H)$. But then sources and sinks of $H$ (which, by the definition of $H$, have degree $1$) will never appear in a vertex-separator in $u(H)$ (by the definition of undirected branch-width) and hence $\bw(u(H)) = \dbw(H)$. The result follows since, by Lemma 
\ref{lemma:dbw_invariant_source_sink_identification}, $\dbw(H) = \dbw(D)$.
\end{proof}

\subsection{Butterfly minors and directed topological minors.}\label{subsec:minors} 
Two directed analogues of the minor relation are \emph{butterfly minors} and \emph{directed topological minors}. The directed topological minor relation (introduced in~\cite{ganian2010arethere}) is a less-restrictive variant of the better-known notion of a butterfly minor which was introduced in~\cite{dtw}. We show that although directed branch-width is not closed under topological minors, it is closed under butterfly minors. 

\begin{definition}
The \emph{digraph $D/\dir{xy}$ obtained by contracting an edge $\dir{xy}$ of a digraph $D$} is the digraph obtained from $D$ by removing the edge $\dir{xy}$ and the vertices $x$ and $y$ and replacing these with a new vertex $v_{\dir{xy}}$ which has in-neighborhood and out-neighborhood equal to $N_D^-(x) \cup N_D^-(y)$ and $N_D^+(x) \cup N_D^+(y)$ respectively in $D/\dir{xy}$.
\end{definition}

Both directed topological minors and butterfly minors are defined by taking subgraphs and contracting edges; what distinguishes them is which edges are deemed `contractible'. Directed topological minors allow only the contraction of \emph{$2$-contractible edges} (defined below).

\begin{definition}[\cite{ganian2010arethere}]\label{def:2-contractible} 
Let $D$ be a digraph and let $V_3(D)$ denote the subset of vertices incident with at least $3$ edges in $D$. An edge $a = \dir{xy}$ is \emph{$2$-contractible} in $D$ if 
\begin{itemize}
    \item $\{x,y\} \not \subseteq V_3(D)$
    \item $\dir{yx} \in E(D)$ or there does \textbf{not} exist a pair of vertices $(w,z)$ in $V_3(D)$, possibly $w=z$, such that $x$ can reach $w$ in $D-\dir{xy}$ and $z$ can reach $y$ in $D-\dir{xy}$.
\end{itemize}
\end{definition}

Butterfly minors, on the other hand, only allow the contraction of \emph{butterfly edges}.

\begin{definition}
Let $\dir{xy}$ be an edge in a digraph $D$. If $x$ has out-degree $1$ and $y$ has in-degree $1$, then we call $\dir{xy}$ a \emph{butterfly edge}.
\end{definition}

Note that the contraction of a butterfly edge never creates new directed paths that were not otherwise present. On the other hand, this might happen when contracting a $2$-contractible edge (for example one joining a source to a sink).

Having defined these two kinds of edges, we can define directed topological minors and butterfly minors. 

\begin{definition}[\cite{ganian2010arethere, dtw}]\label{def:dir_topo_minor}
Let $H$ and $D$ be digraphs. 
\begin{itemize}
\item $H$ is a \emph{directed topological minor} of $D$ if $H$ can be obtained from a subgraph of $D$ via a sequence of contractions of \emph{$2$-contractible} edges.~\cite{ganian2010arethere}
\item $H$ is a \emph{butterfly-minor} of $D$ if it can be obtained from a subgraph of $D$ via sequence of contractions of \emph{butterfly edges}~\cite{dtw}. 
\end{itemize}
\end{definition}

The next theorem shows directed branch-width is not closed under directed topological minors. In contrast, we point out that all the tree-width-inspired measures are indeed closed under directed topological minors~\cite{ganian2010arethere}. 

\begin{theorem}\label{thm:dbw_not_topo_minor_closed}
There does not exists any function $g: \mathbb{N} \rightarrow \mathbb{N}$ such that, for every directed topological minor $H$ of a digraph $J$, we have $\dbw(H) \leq g(\dbw(J))$.
\end{theorem}
\begin{proof}
We shall construct a set of graphs $\{D_{n}: n \in \mathbb{N}\}$ of bounded directed branch-width and show that there is a set $\{\Delta_{n}' : n \in \mathbb{N}\}$ of directed topological minors of elements of $\{D_{n}: n \in \mathbb{N}\}$ which has unbounded directed branch-width.

Define $D_n$ by starting from the $n$-vertex directed cycle $C_n$ and proceeding as follows:
\begin{itemize}
    \item add $n$ sources $s_1, \ldots, s_n$, each one adjacent to every vertex of $C_n$
    \item add the vertices $a_1,\ldots, a_n$ and $b_1, \ldots, b_n$ and, for each $i \in [n]$, add the edges $\dir{s_ia_i}$ and $\dir{b_ia_i}$ (see Figure \ref{fig:dbw_not_topo_minor_closed}).
\end{itemize}
We will now show that $\dbw(D_n) \leq 3$. Let $D_n'$ be the source-sink-split of $D_n$ and note that $u(D_n')$ is a cycle with $n$ pendant edges at each vertex together with some isolated edges (corresponding to the edges $\dir{s_ia_i}$ and $\dir{b_ia_i}$). Thus, by inspection, $u(D_n')$ has tree-width $2$ (since deleting a single vertex from the cycle in $u(D_n')$ yields a forest). By Theorem \ref{thm:bounded_tw_bounded_bw}, this implies that $\bw(u(D_n')) \leq 3$ which, by Corollary \ref{corollary:dbw_bw_number_of_sources_and_sinks}, implies that $\dbw(D_n') \leq 3$ and hence, by Lemma \ref{lemma:dbw_invariant_source_sink_identification}, that $\dbw(D_n) \leq 3$.

Now we will obtain a set $\{\Delta_n : n \in \mathbb{N}\}$ having unbounded directed branch-width such that $\Delta_n$ is a directed topological minor of $D_{n}$. Note that each edge $\dir{s_ia_i}$ is $2$-contractible since, for each $i$, $a_i$ has degree $2$ and since the only vertex that can reach $a_i$ in $D_{n} - \dir{s_ia_i}$ (namely $b_i$) has degree $1$. Thus, by contracting every edge $\dir{s_ia_i}$, we construct the digraph  $\Delta_{n}$ which is a directed topological minor of $D_{n}$. Now consider the digraph $\Delta_{n}'$ obtained from $\Delta_{n}$ by identifying all of the sources $b_1, \ldots, b_n$ into a single source $b$ (see Figure \ref{fig:dbw_not_topo_minor_closed}). By Lemma \ref{lemma:dbw_invariant_source_sink_identification} we know that $\dbw(\Delta_{n}') = \dbw(\Delta_{n})$ so we shall conclude the proof by showing that $\{\Delta_{n}': n \in \mathbb{N}\}$ has unbounded directed branch-width.
Since $u(\Delta_{n}')$ contains the balanced complete bipartite graph $K^{n,n}$ as a subgraph and since $\tw(K^{n,n}) = n$~\cite{Diestel2010GraphTheory}, we deduce (via Theorem \ref{thm:bounded_tw_bounded_bw}) that \[\bw(u(\Delta_{n}')) \geq 2( \tw(K^{n,n})/3 -1) = 2n /3 - 2.\] Finally, since $\Delta_{n}'$ has at exactly one source and no sinks, we deduce by Corollary \ref{corollary:dbw_bw_number_of_sources_and_sinks} that \[\dbw(\Delta_{n}') \geq \bw(u(\Delta_{n}')) - 1 \geq 2 n /3 - 3. \]
Thus we conclude that $\{D_{n}: n \in \mathbb{N}\}$ has bounded directed branch-width and $\{\Delta_{n}' : n \in \mathbb{N}\}$ has unbounded directed branch-width as required.
\end{proof}

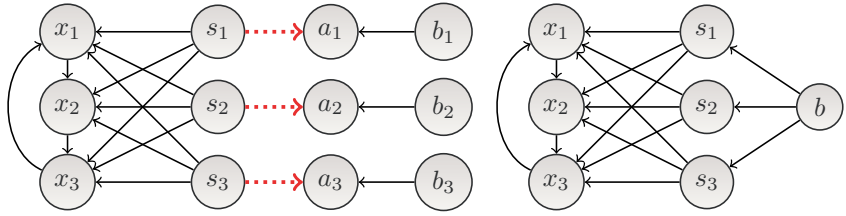
\begin{figure}[h]
\centering
\begin {tikzpicture}[-latex ,auto ,node distance =1 cm and 1cm ,on grid ,
semithick , state/.style ={ circle ,top color =softGray , bottom color = creamWhite ,
draw, softBlack , text=softBlack , minimum width =0.15 cm}]

\node[state] (X1) [] {$x_1$};
\node[state] (X2) [below =of X1] {$x_2$};
\node[state] (X3) [below =of X2] {$x_3$};

\node[state] (S1) [right =2cm of X1] {$s_1$};
\node[state] (S2) [below =of S1] {$s_2$};
\node[state] (S3) [below =of S2] {$s_3$};

\node[state] (A1) [right =1.5cm of S1] {$a_1$};
\node[state] (A2) [below =of A1] {$a_2$};
\node[state] (A3) [below =of A2] {$a_3$};

\node[state] (B1) [right =1.5cm of A1] {$b_1$};
\node[state] (B2) [below =of B1] {$b_2$};
\node[state] (B3) [below =of B2] {$b_3$};

\path (X1) edge [->] (X2);
\path (X2) edge [->] (X3);
\path (X3) edge [->, bend left = 65 ] (X1);

\path (S1) edge [->] (X1);
\path (S1) edge [->] (X2);
\path (S1) edge [->] (X3);

\path (S2) edge [->] (X1);
\path (S2) edge [->] (X2);
\path (S2) edge [->] (X3);

\path (S3) edge [->] (X1);
\path (S3) edge [->] (X2);
\path (S3) edge [->] (X3);

\path (S1) edge [->, color = brick, line width=0.50mm, dotted] (A1);
\path (S2) edge [->, color = brick, line width=0.50mm, dotted] (A2);
\path (S3) edge [->, color = brick, line width=0.50mm, dotted] (A3);

\path (B1) edge [->] (A1);
\path (B2) edge [->] (A2);
\path (B3) edge [->] (A3);


\node[state] (X1') [right = 1.5cm of B1] {$x_1$};
\node[state] (X2') [below =of X1'] {$x_2$};
\node[state] (X3') [below =of X2'] {$x_3$};

\node[state] (S1') [right =2cm of X1'] {$s_1$};
\node[state] (S2') [below =of S1'] {$s_2$};
\node[state] (S3') [below =of S2'] {$s_3$};

\node[state] (B) [right =1.5cm of S2'] {$b$};

\path (X1') edge [->] (X2');
\path (X2') edge [->] (X3');
\path (X3') edge [->, bend left = 65] (X1');

\path (S1') edge [->] (X1');
\path (S1') edge [->] (X2');
\path (S1') edge [->] (X3');

\path (S2') edge [->] (X1');
\path (S2') edge [->] (X2');
\path (S2') edge [->] (X3');

\path (S3') edge [->] (X1');
\path (S3') edge [->] (X2');
\path (S3') edge [->] (X3');

\path (B) edge [->] (S1');
\path (B) edge [->] (S2');
\path (B) edge [->] (S3');

\end{tikzpicture}
\caption{
    \textbf{Left}: the graph $D_{3}$ defined in the proof of Theorem \ref{thm:dbw_not_topo_minor_closed} (the relevant $2$-contractible edges are drawn red and dotted). \textbf{Right}: the graph $\Delta'_{3}$. 
}\label{fig:dbw_not_topo_minor_closed}
\end{figure}

Now we turn our attention to butterfly minors and show that directed branch-width is a butterfly-minor-closed parameter. 

\begin{theorem}
If $H$ is a butterfly minor of a digraph $D$, then $\dbw(H) \leq \dbw(D)$.
\end{theorem}
\begin{proof}
Since directed branch-width is subgraph-closed (by Lemma \ref{lemma:dbw_subgraph_closed}) it suffices to show that it is not increased under contractions of butterlfy edges. We will show that, given any directed branch decomposition $(T, \beta)$ of $D$, we can construct a directed branch decomposition of $D /_{\dir{xy}}$ of width at most that of $(T, \beta)$.

Let $\dir{xy}$ be a butterfly edge in $D$ and let $\omega$ be the vertex of $D/_{\dir{xy}}$ created by the contraction of $\dir{xy}$. Let $\gamma: E(D) \setminus \{\dir{xy}\} \to E(D/_{\dir{xy}})$ be the bijection defined for any edge $\dir{st}$ in $E(D) \setminus \{\dir{xy}\}$ as: 
\[
 \gamma(\dir{st}) := 
            \begin{cases}
                \dir{s \omega} \text{ if } t = x\\
                \dir{\omega t} \text{ if } s = y\\
                \dir{st} \text{ otherwise}.
            \end{cases}
\]
Note that $\gamma$ is well-defined since $x$ and $y$ are respectively a sink and a source in $E(D) \setminus \{\dir{xy}\}$ (by the definition of butterfly edge). 

We will use $\gamma$ to construct a directed branch-decomposition $(U, \delta)$ of $D/_{\dir{xy}}$ from $(T, \beta)$ as follows. Let $U$ be the inclusion-wise minimal subtree of $T$ connecting all leaves of $T$ which are not mapped to $\dir{xy}$ by $\beta$. Let $\delta := \gamma \circ (\beta|_{\ell(U)})$ where $\beta|_{\ell(U)}$ denotes the restriction of $\beta$ to the set $\ell(U)$ of leaves of $U$ (note that, by the definition of $\gamma$, $\delta$ is a bijection between the leaves of $U$ and the edges of  $D/_{\dir{xy}}$).

Let $e$ be an edge in $U$ (and note that $e \in E(T)$ also). Let $X_e^{(T, \beta)}$ be the bidirected separator at $e$ in $(T, \beta)$ and let $X_e^{(U, \delta)}$ be the bidirected separator at $e$ in $(U, \delta)$. By the definition of $(U, \delta)$, we have that
\[
    X_e^{(U, \delta)} = 
    \begin{cases}
        \{\omega\} \cup (X_e^{(T, \beta)} \setminus \{x,y\}) \text{ if } \{x,y\} \cap X_e^{(T, \beta)} \neq \emptyset \\
        X_e^{(T, \beta)} \text{otherwise}
    \end{cases}
\]
and hence that $|X_e^{(U, \delta)}| \leq |X_e^{(T, \beta)}|$. Since this is true for any edge $e$, it implies that $\dbw(D/_{\dir{xy}}) \leq \dbw(D)$ as required.
\end{proof}

\subsection{Comparison to other digraph width measures.}\label{subsec:dbw_comparison_to_other_measures}
All of the ``tree-width-inspired'' width measures (such as \emph{directed tree-width}, \emph{\texttt{DAG}-width}, \emph{$D$-width} and \emph{Kelly-width}) are bounded on the class of all \texttt{DAGs}~\cite{kreutzer2018}. In contrast we show that the class of all \texttt{DAGs} has \emph{unbounded} directed branch-width.

\begin{theorem}\label{theorem:DAGs_unbounded_dbw}
For any $n \in \mathbb{N}$, there exists a directed acyclic graph $D_n$ with $\dbw(D_n) \geq n$.
\end{theorem}
\begin{proof}
Let $\kappa = 3(n+2)/2 - 1$ and $D_n$ be the acyclic digraph obtained by orienting all edges of the $\kappa \times \kappa$ grid north-east, i.e. the digraph 
\[D_n := \Bigl ( [\kappa]^2, \; \{\dir{(a,b)(c,d)}: (a = c \text{ and } d = b+1) \text{ or } (b = d \text{ and } c = a+1)\} \Bigr ). \]
Recall that, since $u(D_{n})$ is a square grid, it has tree-width $\kappa = 3(n+2)/2 - 1$~\cite{Diestel2010GraphTheory}. Thus we can invoke Theorem \ref{thm:bounded_tw_bounded_bw} to deduce that $\bw(u(D_n)) \geq n + 2$. Since $D_n$ has exactly one source and one sink (the south-west and north-east corners), we know (by Corollary \ref{corollary:dbw_bw_number_of_sources_and_sinks}) that $\dbw(D_n) \geq \bw(u( D_{n} )) -2$, so the result follows.
\end{proof}

Recall that each tree-width-inspired measure (and tree-width itself) can be defined via some variant of a cops-rober game played on either a digraph or an undirected graph respectively. Using Theorem \ref{theorem:DAGs_unbounded_dbw}, we will show that none of these games can be used to characterize directed branch-width. 

The cops-robber variants that can be used to define the tree-width-inspired measures all have the following common structure. The Sheriff can chose to add new cops to the board or move some cops that have been previously placed while the Villain can only move their piece. Each round follows this sequence of events: the Sheriff announces which cops they intend to move and where they intend to move them to, then they remove any cops that are involved in this move from the board; before the cops are placed back on the board, the Villain makes their move; finally the Sheriff completes their previously-announced move. 

Playing on a (directed) graph $D$, the state of the game after round $\rho$ is recorded by the pair $(C,r)_{\rho}$ where $C \subseteq V(D)$ is the set of positions of the cops and $r \in V(D)$ is the position of the robber. The initial state of the game is always of the form $(\emptyset, r)_0$ for some vertex $r$ in $D$ and the game terminates when a cop is placed on the vertex occupied by the robber. The Sheriff loses if the game never terminates and wins otherwise. The \emph{width of a play at time $\rho$} is the maximum number of cops that were ever placed at any round up to and including $\rho$. We say that the Sheriff has a \emph{$k$-winning strategy} on a (directed) graph $D$ if they can always win on $D$ with a play of width at most $k$. 

The cops can be moved (or placed) freely from their position to any other vertex in the graph (this is often conceptualized as the cops being able to move ``by helicopter''). Different game variants are distinguished by the knowledge of the Sheriff or the different ways in which a robber can be moved. The first distinction is whether or not the Sheriff knows the position of the robber; we call these games \emph{visible} and \emph{invisible-robber-game} respectively. The second important distinction emerges from how the Villain is allowed to move the robber. For the games used to characterize tree-width-inspired measures, the least restrictive movement pattern for the robber is for it to be movable from position $r$ in a (directed) graph $G$ to any vertex $r'$ which is reachable by a (directed) path from $r$ in $G - C$, where $C$ is the current vertex-set occupied by cops. This is called the \emph{weak reachability game}. We note that we also allow the Villain to pass their turn without moving their piece (this is referred to as an \emph{inert robber game}~\cite{kreutzer2018}). 

The games we have described can be used to characterize both tree-width and any tree-width inspired measure. An undirected graph $G$ has tree-width at most $k$ if and only if the Sheriff has a $(k+1)$-winning-strategy for the visible cops-robbers game on $G$~\cite{cygan2015parameterized}. A similar statement can be made about any tree-width-inspired measure $\mu$: there exists a cops-robber game $\Gamma_\mu$ such that the Sheriff has a $k$-winning-strategy for $\Gamma_\mu$ on a digraph $D$ if and only if $\mu(D) \leq k$~\cite{kreutzer2018}.

Note that the most general game variant is the \emph{invisible-robber inert-weak-reachability cops-robber game}. This follows since: (1) every strategy playable with a robber that cannot remain inert is also playable with one that can, (2) if the Sheriff can win the invisible-robber game, then they can also win in the visible version (by closing their eyes).

\begin{corollary}\label{corollary:no_cops_robbers}
The Sheriff has a $1$-winning strategy for the invisible-robber inert-weak-reachability cops-robber game on any directed acyclic graph.
\end{corollary} 
\begin{proof}
Let $D$ be an acyclic digraph and let $v_1, \ldots, v_n$ be a topological ordering of its vertices. The Sheriff's strategy on $D$ is to place one cop at $v_1$ and then move it at every round to the next next vertex in the topological ordering. Note that, if the robber is moved at any round $\rho$ from some vertex $v_i$ to $v_j$, then it must be that $j \geq i$. Thus, since the cop will eventually reach $v_n$, the robber cannot escape indefinitely. 
\end{proof}

By Theorem \ref{theorem:DAGs_unbounded_dbw}, we know that the class of all \texttt{DAGs} has unbounded directed branch-width. Thus Corollary \ref{corollary:no_cops_robbers} implies the following result. 

\begin{corollary}
There is a family of digraphs of unbounded directed branch-width for which the Sheriff has a $1$-winning strategy for the invisible-robber inert-weak-reachability cops-robber game.
\end{corollary}

Undirected graph classes of bounded branch-width also have bounded rank-width~\cite{CorneilRoticsTwCw, oum2017rank}. We will show that this is not true for the directed analogues of these measures: boundedness of directed branch-width does not imply boundedness of bi-cut-rank-width (nor does it imply the boundedness of any rank-width inspired measure, by Theorem \ref{thm:clique_width_equivalence}). To do so, we fist recall a known result about the rank-width of grids. 

\begin{theorem}[\cite{jelinek2010rank}]\label{thm:rw_of_grids}
The class of all undirected grids has unbounded rank-width.
\end{theorem} 

For directed graphs, the bi-cut-rank-width version of Theorem \ref{thm:rw_of_grids} follows by the following theorem.

\begin{theorem}[\cite{gurski2016directedNLC, kreutzer2018}]\label{thm:bcrw_grids}
For any digraph $D$, $\bcrk(D) \geq \rw(u(D))$.
\end{theorem}

Thus Theorems \ref{thm:rw_of_grids} and \ref{thm:bcrw_grids} allow us to show that boundedness of directed branch-width does not imply boundedness of bi-cut-rank-width. 

\begin{theorem}
There does not exist a function $g: \mathbb{N} \to \mathbb{N}$ such that, for any digraph $D$, $\bcrk(D) \leq g(\dbw(D))$.
\end{theorem}
\begin{proof}
By Corollary \ref{corollary:grid_dbw}, there is an infinite set $\vec{\Gamma}$ of digraphs given by acyclic orientations of the square grids such that $\vec{\Gamma}$ has bounded directed branch-width. However, by Theorems \ref{thm:rw_of_grids} and \ref{thm:bcrw_grids}, we know that $\vec{\Gamma}$ has unbounded bi-cut-rank-width. 
\end{proof}


\section{Conclusion and open problems.}\label{sec:conclusion}
We introduce a new digraph width measure called directed branch-width by generalizing to the setting of digraphs a characterization of classes of bounded tree-width as classes of graphs whose line-graphs have bounded rank-width. Furthermore, we proved that the classes of digraphs of bounded directed branch-width are exactly those classes whose \textit{directed} line-graphs have bounded bi-cut-rank-width. 

From an algorithmic perspective (see the Appendix) directed branch-width occupies the middle ground  between undirected tree-width and its previously known directed analogues: although classes of bounded directed branch-width need not have bounded underlying tree-width, we can nonetheless exploit close connections between our new parameter and underlying tree-width to place many problems in $\fpt$ when parameterized by directed branch-width. We showed that the model-checking problem for a specific subset of $\MSO_2$-formulae (namely those for which the corresponding model-checking problem is source-sink invariant) is also in $\fpt$ parameterized by directed branch-width. Furthermore, we note that the directed analogues of the Hamilton Path, and Max-Cut problems are in $\fpt$ parameterized by directed branch-width. We remark, however, that, although the algorithmic applications of directed branch-width that we studied here can often be explained by the algorithmic power of its undirected counterpart, they are nonetheless interesting given the fact that both of the last two problems we just mentioned are $\mathbf{W[1]}$-hard when parameterized by any tree-width-inspired or rank-width-inspired measure~\cite{lampis2008algorithmic, fomin2010intractability}. Thus it is clear that the quest for algorithmically powerful directed analogues of tree-width is still ongoing and we believe that there is still much insight to be gained by exploring width measures such as directed branch-width which occupy a middle ground between underlying tree-width and the other known `tree-width-' and `rank-width-inspired' measures.

Directed branch-width opens new doors for the study of digraph connectivity which we describe now. This is partly due to a connection with \emph{tangles} which are objects allowing one to investigate global connectivity properties of graphs (see~\cite{robertsonX} for a definition). It is known that layouts of a \emph{symmetric sub-modular} function $f$ over a set $S$ are always dual to $f$-tangles over $S$~\cite{amini2009submodular, diestel2014unifying, diestel2017tangle, robertsonX}. In particular this means that it would be possible to define a notion of \emph{directed tangle} which is dual to directed branch-width. 


We point out that Giannopoulou, Kawarabayashi, Kreutzer and Kwon~\cite{kreutzerTangle} also introduced a notion under the name of `directed tangle'. However, the two notions are completely incomparable: while our definition of `directed tangle' is defined with respect to a connectivity function that is symmetric and submodular, the connectivity function that they use is not\footnote{in their notation, a \emph{directed separation} in a digraph $D$ is a pair $(L,R)$ of \emph{vertex subsets} of $D$ such that $L \cup R = V(D)$ and such that there is no \emph{pair} of edges $\dir{\ell_1r_1}$ and $\dir{r_2\ell_2}$ which `cross $(L,R)$ in opposite directions' in $D$ (i.e. s.t. $\{\ell_1, \ell_2\} \subseteq L$ and $\{r, r_2\} \subseteq R$) }. Thus, comparing our notion of `directed tangle' to theirs, is another interesting avenue for further research.

\appendix
\section{Algorithmic aspects of directed branch-width.}\label{sec:algorithms}
In this section, after reviewing some basic facts from parameterized complexity, we will first show that directed branch-width is computable in $\fpt$-time when parameterized by directed branch-width. Then we will see that, although classes of bounded directed branch-width need not have bounded underlying tree-width, we can nonetheless exploit close connections between our new parameter and underlying tree-width to place many problems (such as Directed Hamilton Path and Directed MaxCut) in $\fpt$ when parameterized by directed branch-width. 

\subsection{Algorithmic background}
A parameterized problem is \emph{slice-wise polynomial tractable} if there exists an algorithm $A$ which decides an instance $I$ with parameter $k$ in time at most $f(k)|I|^{g(k)}$, where $f$ and $g$ are computable functions. If $g$ is a constant function, then we call $A$ a \emph{fixed-parameter algorithm} with respect to the parameter $k$ and we call any parameterized problem admitting such an algorithm \emph{fixed parameter tractable}. The class of all slice-wise polynomial tractable problems is denoted $\xp$ and the class of all fixed parameter tractable problems is denoted $\fpt$. For other notions related to parameterized complexity (such as parameterized reductions), we refer the reader to~\cite{cygan2015parameterized, flum2006parameterized}.

\emph{Monadic second order logic} (denoted $\MSO$) is the fragment of second order logic which allows second-order quantification only over unary predicates. For graphs it is useful to distinguish two kinds of $\MSO$-logic: $\MSO_1$-logic allows second order quantification only over vertex-sets while $\MSO_2$-logic allows second order quantification both over vertex-sets and edge-sets. 

It is known that any problem expressible in $\MSO_2$-logic can be decided in linear time on classes of bounded tree-width.

\begin{theorem}[Courcelle's Theorem,~\cite{courcelle1990monadic}]\label{thm:courcelle}
There is a fixed-parameter algorithm which, given a graph $G$ of tree-width $k$ and an $\MSO_2$-formula $\phi$, decides whether $G \models \phi$ in linear time when both $k$ and $|\phi|$ are bounded by constants.
\end{theorem}

The authors of~\cite{ganian2010arethere} investigate the question of why there are no known digraph width measures that are as algorithmically successful as tree-width in terms of parameterizations for $\fpt$ algorithms. They prove that unless $\np \subseteq \mathbf{P} \: \setminus \mathbf{poly}$, it is impossible to obtain an $\xp$ algorithm for the $\MSO_1$-model-checking problem parameterized by any digraph width measure similar to tree-width-inspired measures. Their result (Theorem \ref{thm:ganian_thm}) mentions a digraph containment relation called \emph{directed topological minor relation}~\cite{ganian2010arethere} which we define formally in  Section \ref{sec:dbw_properties}.

\begin{theorem}[\cite{ganian2010arethere}]\label{thm:ganian_thm}
A digraph width measure $\delta$ is said to be tree-width bounding if there exists a computable function $b$ such that for every digraph $D$ with $\delta(D) \leq k$, we have $\tw(u(D)) \leq b(k)$. If $\delta$ is a digraph width measure which is neither tree-width-bounding nor closed under taking directed topological minors, then, \emph{unless} $\np \subseteq \mathbf{P} \: \setminus \mathbf{poly}$, there is no $\xp$ algorithm for the $\MSO_1$-model-checking problem parameterized by $\delta$.
\end{theorem}

We point out that Theorem \ref{thm:ganian_thm} excludes the existence of an $\xp$ algorithm: this is a much stronger claim than excluding fixed-parameter-tractability.

\subsection{Computing directed branch-width}
Here we will show that, despite being $\np$-complete, the problem of determining whether a digraph has branch-width at most $k$ can be solved in linear time for any constant $k$.

\noindent We begin with a formal statement of the problem.

\begin{framed}
\noindent
    \textbf{\textsc{Directed-Branch-Width}.} \\
    \textbf{Input:} a digraph $D$ and a positive integer $k$.\\
    \textbf{Question:} does $D$ have directed branch-width at most $k$?
\end{framed}

It is known that deciding whether the branch-width of an undirected graph is at most $k$ is an $\np$-complete problem~\cite{seymour1994call_comput_bw_NP_hard}. Since directed branch-width agrees with undirected branch-width on bi-directed orientations of undirected graphs (Corollary \ref{corollary:dbw_bidirected_orientations}), it follows that \textsc{Directed-Branch-Width} is $\np$-hard. Furthermore, given a directed branch decomposition $(T, \beta)$ of a digraph $D$, one can check in polynomial time the order of any edge of $T$. Since $T$ has $O(|E(D)|)$ edges, \textsc{Directed-Branch-Width} is in $\np$; thus we have shown the following result.

\begin{theorem}
\textsc{Directed-Branch-Width} is $\np$-complete. 
\end{theorem}

A natural next question is whether computing the directed branch-width of a digraph $D$ is in $\fpt$ parameterized by $k$. This is known for the undirected case.

\begin{theorem}[\cite{berthome2013unified, bodlaender1997constructive}]\label{thm:bw_bodlaender}
There is a linear-time algorithm which decides for a fixed $k$ whether an undirected graph has branch-width at most $k$.
\end{theorem}

We can use Theorem \ref{thm:bw_bodlaender} and Lemma \ref{lemma:dbw_undir_lower_bound} to show that deciding whether a digraph has directed branch-width at most $k$ is in $\fpt$ parameterized by $k$.

\begin{theorem}\label{thm:FPT_construction}
\textsc{Directed-Branch-Width} is in $\fpt$ parameterized by $k$.
\end{theorem}
\begin{proof}
Letting $D'$ be the source-sink-split of $D$, we know (Lemma \ref{lemma:dbw_undir_lower_bound}) that $\bw(u(D')) = \dbw(D)$. A source-sink-split can be obtained in polynomial time: for each source (or sink) $x$ with $N(x) = \{y_1, \dots, y_\eta\}$, replace $x$ with sources (or sinks) $z_1, \ldots, z_\eta$ such that $z_i$ is adjacent to $y_i$. Since $|V(D')|$ is $O(|V(D)|^2)$, we can use the $\fpt$ algorithm of Theorem \ref{thm:bw_bodlaender} to determine whether the undirected branch-width of $u(D')$ is at most $k$ and hence compute whether the directed branch-width of $D$ is at most $k$ in time linear in $|E(D')|$ (which equals $|E(D)|$). 
\end{proof}

\subsection{Parameterizations by directed branch-width}
As we mentioned in Section \ref{sec:Intro}, despite their great success in some cases, the tree-width and rank-width-inspired measures still face considerable shortcomings when compared to the algorithmic power of undirected tree-width. In particular recall that the directed versions of the Hamilton Path and Max-Cut problems are $\mathbf{W[1]}$-hard when parameterized by any rank-width or tree-width-inspired measure~\cite{fomin2010intractability,lampis2008algorithmic}.

\begin{framed}
\noindent
\textbf{\textsc{D-Hamilton-Path}.}\\
\textbf{Input:} a digraph $D$.\\
\textbf{Question:} does $D$ contain a directed Hamiltonian Path?
\end{framed}

\begin{framed}
\noindent
\textbf{\textsc{D-Max-Cut}.} \\
\textbf{Input:} a digraph $D$ and an integer $k$. \\
\textbf{Question:} is there a vertex partition of $D$ of directed order at least $k$?
\end{framed}

Here we will show that, although classes of bounded directed branch-width need not have bounded underlying tree-width, we can nonetheless exploit close connections between our new parameter and underlying tree-width to place \textsc{D-Hamilton-Path} and \textsc{D-Max-Cut} in $\fpt$ when parameterized by directed branch-width. We will do so by first showing some general results providing sufficient conditions which guarantee that a problem is in $\fpt$ parameterized by directed branch-width. 

One notion important for our sufficient conditions will be that of invariance under source-sink-identifications. 

\begin{definition}
We call a decision problem $\Pi$ \emph{source-sink-invariant} if, given any two digraphs $D$ and $H$ equivalent under source-sink-identifications, $D$ is a yes-instance for $\Pi$ if and only if $H$ also is.
\end{definition}

Note that, while problems such as Directed Feedback Vertex Set, Acyclic Coloring and \textsc{D-Max-Cut} are source-sink-invariant (see the following lemma for \textsc{D-Max-Cut}), the \textsc{D-Hamilton-Path} problem is not. To see this, consider the digraph $J = (\{a,b,c\}, \{\dir{ab}, \dir{cb}\})$. Clearly $J$ has no directed Hamiltonian path; however, after performing all possible source-sink-identifications, it does.

\begin{lemma}\label{lemma:max_cut_identify_terminating}
\textsc{D-Max-Cut} is source-sink-invariant. 
\end{lemma}
\begin{proof}
We will first show that, in any instance, there is a vertex partition of maximum order with all sources on the left-hand-side and all sinks on the right-hand-side. Then we will show that identifying the sources or the sinks in any such partition does not change its order.

Take any vertex partition $(V(D) \setminus X, X)$ of any digraph $D$. Recall that the order of $(V(D) \setminus X, X)$ is the number of edges in the edge separator $S_{(V(D) \setminus X, X)}^E$ consisting of all edges starting in $V(D) \setminus X$ and ending in $X$. Thus, if there is a source $s$ in $X$, then no edge incident with $s$ will be in $S_{(V(D) \setminus X, X)}^E$. Hence, letting $Y := X \setminus \{s\}$, the order of $(V(D) \setminus Y, Y)$ is at least that of $(V(D) \setminus X, X)$. Furthermore, a symmetric argument shows that moving all sinks to the right-hand-side of a vertex partition can only increase the order of a partition. Thus we have shown that, in every digraph $D$, there exists a vertex partition $(V(D) \setminus Z, Z)$ of maximum order with all sources in $V(D) \setminus Z$ and all sinks in $Z$. 

To conclude the proof, we consider what happens when we identify two sources in $D$ (the proof for identifications of sinks is symmetric). Let $s_1$ and $s_2$ be two sources in $V(D) \setminus Z$ and, for $i \in [2]$, let $F_i$ be the set of all edges starting from $s_i$ and ending in $Z$. Now consider the digraph $D'$ obtained from $D$ by identifying $s_1$ and $s_2$ into a new source $\gamma$. Since the set of edges starting from $\gamma$ and ending in $Z$ is precisely $F_1 \cup F_2$, it follows that $|S_{(V(D') \setminus Z, Z)}^E| = |S_{{(V(D) \setminus Z, Z)}}^E|$.
\end{proof}

We now show that if a source-sink-invariant problem is in $\fpt$ parameterized by underlying tree-width, then it also is when parameterized by directed branch-width. 

\begin{theorem}\label{thm:source-sink-alg}
If $\Pi$ is a source-sink-invariant problem which is in $\fpt$ parameterized by underlying tree-width, then $\Pi$ is in $\fpt$ parameterized by directed branch-width.
\end{theorem}
\begin{proof}
First obtain the source-sink-split $D$ of the input digraph $D'$ and note that this can be done in polynomial time (recall that $|V(D)|$ is $O(|V(D')|)$). Since $\bw(u(D)) = \dbw(D')$ (by Lemma \ref{lemma:dbw_undir_lower_bound}) we know that $\tw(u(D)) \leq 3 k /2 $ (by Theorem \ref{thm:bounded_tw_bounded_bw}). Since $\Pi$ is source-sink-invariant, $D$ will be a yes-instance if and only if $D'$ also is. Thus the result follows since we can apply the $\fpt$ algorithm to solve $\Pi$ on $D$.
\end{proof}

Theorem \ref{thm:source-sink-alg} allows us to deduce a algorithmic meta-theorem parameterized by directed branch-width by leveraging Courcelle's Meta-theorem for tree-width (Theorem \ref{thm:courcelle}). In particular Theorems \ref{thm:source-sink-alg} and \ref{thm:courcelle} imply that the model-checking problem for the subset of source-sink-invariant formulae in $\MSO_2$-logic is in $\fpt$ parameterized by directed branch-width. 

\begin{corollary}\label{corollary:our_courcelle}
Let $\phi$ be a formula in the $\MSO_2$ logic of graphs. If, for any digraph $D'$ equivalent to $D$ under source-sink identification, we have that $D' \models \phi$ if and only if $D \models \phi$, then there is an $\fpt$ algorithm parameterized by $\dbw(D) + |\phi|$ which decides whether $D$ models $\phi$. 
\end{corollary}

We note that Corollary \ref{corollary:our_courcelle} does not immediately imply the existence of an $\fpt$ algorithm for \textsc{D-Hamilton-Path} since this is not a source-sink-invariant problem. However notice that this problem is tractable if the input digraph has more than one source or more than one sink (such a digraph cannot be Hamiltonian). Thus, for problems that are tractable on inputs with more than some constant number of sources or sinks, we know that either the problem is tractable or we know that the underlying tree-width is bounded in terms of the directed branch-width. More generally, the following result shows that, for any fixed $\gamma$, problems which are polynomial-time solvable on digraphs with more than $\gamma$ sources or sinks and which are also in $\fpt$ parameterized by underlying tree-width (such as \textsc{D-Hamilton-Path}), belong to $\fpt$. 

\begin{corollary}\label{corollary:few_sources}
Let $\Pi$ be a decision problem on digraphs which is in $\fpt$ parameterized by underlying tree-width. If there exists a polynomial $p$ and integer $\gamma$ such that $\Pi$ can be decided in time $p(|D|)$ on any digraph with at least $\gamma$ sources or at least $\gamma$ sinks, then $\Pi$ is in $\fpt$ parameterized by directed branch-width.
\end{corollary}
\begin{proof}
Let $D$ be a digraph and let $S$ and $T$ be the sets of all sources and sinks of $D$. Check (in time $O(|E(D)|)$) whether $|S| \geq \gamma$ or $|T| \geq \gamma$. If this is the case, then we decide $\Pi$ in time $p(|D|)$. Otherwise, by Corollary \ref{corollary:dbw_bw_number_of_sources_and_sinks}, we know that $\bw(u(D)) \leq \dbw(D) + |S| + |T| \leq \dbw(D) + 2\gamma$. Thus, since the $\Pi$ is in $\fpt$ parameterized by underlying tree-width and since $\tw(u(D)) \leq 3(\dbw(D) + 2\gamma)/2 -1$ (by Theorem \ref{thm:bounded_tw_bounded_bw}), the result follows.
\end{proof}

Finally we turn our attention to \textsc{D-Hamilton-Path} and \textsc{D-Max-Cut}. Recall that, when parameterized by underlying tree-width, these problems are in $\fpt$.

\begin{theorem}\cite{cygan2015parameterized}\label{thm:known_algs}
For any digraph $D$, \textsc{D-Hamilton-Path} and \textsc{D-Max-Cut} can be solved by algorithms running in time at most $f(\tw(D))|D|$ for a computable function $f$.
\end{theorem}

Thus, by Lemma \ref{lemma:max_cut_identify_terminating}, Theorem \ref{thm:source-sink-alg} and Corollary \ref{corollary:few_sources} we can the following result that avoids the tower of exponentials given by an application of Courcelle's Theorem. 

\begin{corollary}\label{corollary:fast_algs}
For any digraph $D$ with directed branch-width $k$, \textsc{D-Hamilton-Path} and \textsc{D-Max-Cut} can be solved by algorithms running in time at most $f(k)|D|^2$ for a computable function $f$.
\end{corollary}
\begin{proof}
\textsc{D-Max-Cut} is source-sink-invariant by Lemma \ref{lemma:max_cut_identify_terminating}. Thus we can solve it in time linear in the number of vertices of the source-sink-split of any input digraph $D$ by Theorems \ref{thm:known_algs} and \ref{thm:source-sink-alg}. So the result follows for \textsc{D-Max-Cut} since the source-sink-split has $O(|V(D)|^2)$ vertices. Now notice that any digraph with more than one source or more than one sink is a no-instance of \textsc{D-Hamilton-Path}. Thus, by Theorem \ref{thm:known_algs} and Corollary \ref{corollary:few_sources}, the result follows for \textsc{D-Hamilton-Path} as well.
\end{proof}

Since directed branch-width is not closed under directed topological minors, Theorem \ref{thm:ganian_thm} does not rule out the existence of a stronger algorithmic meta-theorem. Thus it is natural to ask whether it is possible to strengthen Corollary \ref{corollary:our_courcelle} to show fixed-parameter tractability of the $\MSO_2$-model-checking problem (i.e. not only for those formulae which are invariant under source-sink identification) parameterized by directed branch-width. We shall show that this is not possible by demonstrating a problem which is $\MSO_2$-expressible but which is $\mathbf{W[1]}$-hard when parameterized by directed branch-width. For this purpose, we define the \emph{Directed $2$-Reachability Edge Deletion problem} (denoted $\dred$). The \emph{$2$-reach of a vertex $x$ in a digraph $D$} is the set $\reach(D,x)$ of all vertices which are reachable from $x$ in $D$ via a directed path with $0$, $1$, or $2$ edges. 

    \begin{framed}
    \noindent
        \textbf{$\dred$.} \\
        \textbf{Input:} a digraph $D$ and two naturals $k$ and $h$ and a vertex $s$ of $D$.\\
        \textbf{Question:} is there an edge-subset $F$ of $D$ of cardinality at most $k$ such that \[|V(D) \setminus \reach(D - F, s)| \geq h?\]
    \end{framed}

\begin{lemma}\label{lemma:DRED_MSO2}
Every instance $(D,k,h,s)$ of the $\dred$ problem is $\MSO_2$-expressible with a formula of length bounded by a function of $k$ and $h$.
\end{lemma}
\begin{proof}
We provide an $\MSO_2$ encoding of the $\dred$ problem via a formula whose length depends only on $k$ and $h$. Let $(D, k, h, s)$ be an instance of the $\dred$ problem. 
A digraph $D$ is encoded as a relational structure with ground set $V(D) \cup E(D)$ and a binary incidence relation. We shall write $E(x,y)$ as shorthand to denote that there is an edge $\dir{xy}$ incident with $x$ and $y$.
For an edge relational variable $F$ and edge-variable $\dir{xy}$ we write $F(\dir{xy})$ to denote that $\dir{xy}$ is in the edge-set $F$. 

Note that there is a directed walk from $s$ to some vertex $t$ with $0$, $1$ or $2$ edges none of which intersect a given edge-set $F$ if and only if at least one of the following formulae holds:
\begin{align*}
\texttt{path}_0(F, t) &:= s = t \text{ or}\\
\texttt{path}_1(F,t) &:= E(s,t) \land \neg F(\dir{st}) \text{ or} \\
\texttt{path}_2(F,t) &:= \exists x \; E(s,x) \land E(x,t) \land \neg F(\dir{sx}) \land \neg F(\dir{xt}).
\end{align*}
Thus, in a digraph $D$, a vertex $t$ is in $\reach(D - F, s)$ if and only if the following formula holds.
\begin{equation}\label{eqn:MSO_formula}
\texttt{canReach}(F, t) := \bigvee_{i \in \{0,1,2\}}\texttt{path}_i(F, t)    
\end{equation}
Using Equation \ref{eqn:MSO_formula}, we determine whether there are at least $h$ vertices which cannot be reached from $s$ in $D - F$ with a path on $0$, $1$ or $2$ edges; this is encoded as follows:
\[ \texttt{unreachable}_{\geq h}(F) := \exists x_1 \ldots \exists x_{h} \bigwedge_{1 \leq i < j \leq h} (x_i \neq x_j) \bigwedge_{i \in [h]} \neg \texttt{canReach}(F, x_i).\]
Next we need to establish a formula that determines whether an edge set $F$ contains at least $k$ edges; we do so as follows: 
\begin{align*}
  \texttt{card}_{\geq k}(F) := \exists e_1 \ldots \exists e_k \;  \bigwedge_{1 \leq i < j \leq k} (e_i \neq e_j) \bigwedge_{1 \leq i \leq k} F(e_i).
\end{align*}
Finally we can use the formulae $\texttt{card}_{\geq k}(F)$ and $\texttt{reach}_{\geq h}(F, s)$ to encode the $\dred$ problem as
\[ \exists F \; \texttt{unreachable}_{\geq h}(F) \land \neg \texttt{card}_{\geq k + 1}(F). \] 
Note that the length of $\texttt{unreachable}_{h}(F)$ is quadratically bounded by a function of $h$ and that the length of $\texttt{card}_{\geq k + 1}(F))$ is quadratically bounded by a function of $k$. Thus the formula-length of the entire encoding depends quadratically only on $k$ and $h$, as desired.
\end{proof}

We will show that $\dred$ is $\mathbf{W[1]}$-hard on graphs of bounded directed branch-width when parameterized by $h$ and $k$. Our proof technique is almost identical to that used by Enright, Meeks, Mertzios, and Zamaraev for an analogous problem in a different setting (temporal graphs)~\cite{temp-edge-del}. For completeness, we give full details here. 

\begin{theorem}[\cite{temp-edge-del}]\label{thm:kitty}
On digraph classes of directed branch-width at most one the $\dred$ problem is $\mathbf{W[1]}$-hard when parameterized by the number $k$ of edges which can be deleted and the number $h$ of non-reached vertices after the edge-deletion. 
\end{theorem}
\begin{proof}
We will describe a parameterized reduction from the $r$-\textsc{CLIQUE} problem (shown to be $\mathbf{W[1]}$-hard in~\cite{DowneyFellows1995}) which asks whether a given undirected graph $G$ contains an $r$-clique, where $r$ is taken as the parameter. Let $(G, r)$ be an $r$-\textsc{CLIQUE}-instance, let $n : = |V(G)|$ and $m := |E(G)|$. From $(G, r)$ construct a $\dred$-instance $(D, r, h, s)$ by defining:
\begin{align*}
V(D) &:= \{ s \} \cup V(G) \cup E(G) \quad \text{ (where } s \text{ is a new vertex not in } G),\\
E(D) &:= \{\dir{sx} : x \in V(G)\} \cup \{ \dir{xe}: x \in V(G) \text{ and } e \in E(G) \text{ s.t. } x \in e\}, \\
h &:=  r + \binom{r}{2}.
\end{align*}  
To see that this is indeed a parameterized reduction, note that $r$ clearly bounds itself, $h$ is a function of $r$ and $|V(D)|$ is polynomial in $|V(G)|$. Finally we claim that $\dbw(D) \leq 1$. To see this, take a source-sink-split $D'$ of $D$ and note that $u(D')$ is a collection of disjoint stars centered at elements of $V(G) \cap V(D)$; thus, by Lemma \ref{lemma:dbw_undir_lower_bound} we have $\dbw(D) = \bw(u(D')) \leq 1$. It remains to be shown that $(G,r)$ is a yes-instance for $r$-\textsc{CLIQUE} if and only if $(D, r, h, s)$ is a yes-instance for $\dred$.

Note that we can assume $r \geq 3$ since otherwise $r$-\texttt{CLIQUE} is trivial. Similarly, we may assume that $m \geq r + \binom{r}{2}$ since otherwise there are at most $r + 3$ vertices of degree at least $r - 1$ and hence we could solve $r$-\texttt{CLIQUE} on $G$ in time $O(r^3)$. 

If $(G, r)$ is a yes-instance, then let $U$ be a vertex-set inducing an $r$-clique in $G$. Then the edge-set $F = \{\dir{su}: u \in U\}$ witnesses that $(D, r, h)$ is a yes-instance since
\begin{align*}
|V(D) \setminus \reach(D - F ,s)| &= (|U| + |E(G[U])|) \; \text{ (by the definition of } F)\\
&=  r + \binom{r}{2} = h \; \text{ (since } U \text{ induces an } r\text{-clique).}
\end{align*}
Conversely, suppose $(D, r, h, s)$ is a yes-instance witnessed by a set $F$ of at most $r$ edges in $D$ such that $|V(D) \setminus \reach(D - F, s)| \geq h$. Let $U$ be the subset of vertices in $V(G) \cap V(D)$ which are incident with an edge in $F$.
Note that $U \subseteq V(G)$ and $|U| \leq r$ by the cardinality of $F$. Since $(D, r, h)$ is a yes-instance we know that
\begin{align*}
|V(D) \setminus \reach(D - F ,s)| &\geq (|U| + |E(G[U])|) \; \text{ (by the definition of } F)\\
&\geq r + \binom{r}{2} \;  \text{ (since } (D, r, h, s) \text{ is a yes-instance)}
\end{align*}
and hence that $|U| + |E(G[U])| \geq r + \binom{r}{2}$. Thus $U$ must induce an $r$-clique in $G$.
\end{proof}

Lemma \ref{lemma:DRED_MSO2} and Theorem \ref{thm:kitty} show that that, unless $\fpt = \mathbf{W[1]}$, the $\MSO_2$-model-checking problem is not in $\fpt$ on classes of bounded directed branch-width parameterized by formula length.


\bibliography{dbw}


\end{document}